%% file: robbery.tex
\documentclass[a4paper,10pt]{amsart}
\usepackage{booktabs}

\input{preamble}

\input{macros}

\title{The Hilbert scheme of infinite affine space and algebraic
K-theory}

\author[M. Hoyois]{Marc Hoyois}
\address{Fakult\"at f\"ur Mathematik\\
Universit\"at Regensburg\\
Universit\"atsstr. 31\\
93040 Regensburg\\
Germany}
\email{\href{mailto:marc.hoyois@ur.de}{marc.hoyois@ur.de}}
\urladdr{\url{http://www.mathematik.ur.de/hoyois/}}
\thanks{M.H., D.N., and M.Y.\ were supported by SFB 1085 ``Higher invariants''}
\thanks{J.J. was supported by NCN grant 2017/26/D/ST1/00913 and by the START
fellowship of the Foundation for Polish Science}
\thanks{B.T. was supported by NSF grant DMS-2054553}

\author[J. Jelisiejew]{Joachim Jelisiejew}
\address{Faculty of Mathematics, Informatics
and Mechanics\\
University of Warsaw\\
Banacha 2\\
02-097 Warsaw\\
Poland}
\email{\href{mailto:jjelisiejew@mimuw.edu.pl}{jjelisiejew@mimuw.edu.pl}}
\urladdr{\url{https://www.mimuw.edu.pl/~jjelisiejew/}}

\author[D. Nardin]{Denis Nardin}
\address{Fakult\"at f\"ur Mathematik\\
Universit\"at Regensburg\\
Universit\"atsstr. 31\\
93040 Regensburg\\
Germany}
\email{\href{mailto:denis.nardin@ur.de}{denis.nardin@ur.de}}
\urladdr{\url{https://homepages.uni-regensburg.de/~nad22969/}}

\author[B. Totaro]{Burt Totaro}
\address{UCLA Mathematics Department\\
Box 951555\\
Los Angeles, CA 90095-1555\\
U.S.A.}
\email{\href{mailto:totaro@math.ucla.edu}{totaro@math.ucla.edu}}
\urladdr{\url{https://www.math.ucla.edu/~totaro/}}

\author[M. Yakerson]{Maria Yakerson}
\address{Institute for Mathematical Research\\
ETH Z\"urich \\
R\"amistr. 101\\  
8092 Z\"urich\\
Switzerland}
\email{\href{mailto:maria.yakerson@math.ethz.ch}{maria.yakerson@math.ethz.ch}}
\urladdr{\url{https://www.muramatik.com}}

\date{\today}

\begin{document}

\begin{abstract} We study the Hilbert scheme $\Hilb_d(\A^\infty)$ from an $\A^1$-homotopical viewpoint and obtain applications to algebraic K-theory. We show that the Hilbert scheme $\Hilb_d(\A^\infty)$ is $\A^1$-equivalent to the Grassmannian of $(d-1)$-planes in $\A^\infty$. We then describe the $\A^1$-homotopy type of $\Hilb_d(\A^n)$ in a range, for $n$ large compared to $d$. For example, we compute the integral cohomology of $\Hilb_d(\A^n)(\C)$ in a range.
	We also deduce that the forgetful map $\FFLAT\to\Vect$ from the moduli stack of finite locally free schemes to that of finite locally free sheaves is an $\A^1$-equivalence after group completion.
	This implies that the moduli stack $\FFLAT$, viewed as a presheaf with framed transfers, is a model for the effective motivic spectrum $\kgl$ representing algebraic K-theory. Combining our techniques with the recent work of Bachmann, we obtain Hilbert scheme models for the $\kgl$-homology of smooth proper schemes over a perfect field. 
\end{abstract}

\maketitle

\parskip 0.2cm

\parskip 0pt
\tableofcontents

\parskip 0.2cm
\vspace{-2em}

\section{Introduction}
\label{sec:intro}

In this paper we analyze the Hilbert scheme of points from the $\A^1$-homotopical perspective, yielding topological information about the Hilbert scheme as well as new geometric models for algebraic K-theory. For simplicity, schemes in the introduction are assumed to be over a perfect field $k$, even though many of our results hold over an arbitrary base scheme.

The Hilbert scheme $\Hilb_d(X)$ classifies zero-dimensional degree $d$ closed subschemes of a fixed scheme $X$. For $X$ a smooth surface, the Hilbert scheme is a smooth variety of dimension $2d$. It has led to rich developments in algebraic geometry, geometric representation theory and string theory~\cite{Beauville__Hilbert_K3, GoettscheICM, Haiman, nakajima_lectures_on_Hilbert_schemes}. For higher-dimensional varieties $X$, the picture is more obscure: the Hilbert scheme is singular (for all $d>3$), has many irreducible components (for large $d$), and almost arbitrary singularities (for $\dim X\geq 16$), see~\cite{Iarrobino, JelisiejewPathologies, Miller_Sturmfels}. Many natural questions remain open~\cite{aimpl}.

Until now, little was known about the topology of $\Hilb_d(X)$ even for $X$ an affine or projective space. We have an explicit basis for the cohomology of $\Hilb_d(\A^2)$ and its motive is pure Tate, thanks to the Bia{\l}ynicki-Birula decomposition for a generic $\G_m$-action on $\A^2$, see \cite[Appendix A]{HPL} or \cite{Ellingsrud_Stromme__On_the_homology}. For $\Hilb_d(\A^n)$ with $n > 2$ this method fails due to the presence of singularities and the motive is not known. The only general results on the topology of those schemes were that $\Hilb_d(\P^n)$ and $\Hilb_d(\A^n)$ are connected, and $\Hilb_d(\P^n)$ is simply connected \cite{HartshorneHilb, Horrocks__Hilb}.

We show that the situation simplifies after stabilization with respect to $n$. We consider the ind-scheme $\Hilb_d(\A^\infty)=\colim_n\Hilb_d(\A^n)$ and construct an $\A^1$-equivalence\footnote{Here and further in the Introduction, we say that a map of presheaves is an $\A^1$-equivalence if it becomes an $\A^1$-equivalence when restricted to affine schemes; in particular, such a map is a motivic equivalence.} in Theorem~\ref{thm:stack-equivalence}:
\begin{equation}\label{eq:Hilb vs Gr}
\Hilb_d(\A^\infty) \simeq \Gr_{d-1}(\A^\infty).
\end{equation}
This implies that for any oriented cohomology theory $A^*$ on $k$-schemes in the sense of \cite[Example 2.1.4]{deglise2011orientation}, such as $l$-adic cohomology, integral cohomology over the complex numbers, motivic cohomology, homotopy invariant K-theory, algebraic cobordism, etc., we have 
\[
A^*(\Hilb_d(\A^\infty)) \simeq A^*(\Spec k)[[c_1,\ldots,c_{d-1}]],\quad |c_i| = i.
\]

We also prove stability results, describing the $\A^1$-homotopy type of $\Hilb_d(\A^n)$ in a range, for $n$ large compared to $d$, see Theorem \ref{thm:Hilb-stability} and Corollary~\ref{cor:motivic-stability}. For example, over the complex numbers,
the homomorphism on integral cohomology 
\[H^*(\Gr_{d-1}(\A^\infty)(\C),\Z) \simeq \Z[c_1,\ldots,c_{d-1}]\to H^*(\Hilb_d(\A^n)(\C),\Z)\]
is an isomorphism in degrees at most $2n-2d+2$.

\begin{rem}
    One can also consider the Hilbert scheme $\Hilb_d(\A^\N)$ where $\A^\N=\Spec \Z[x_1,x_2,\dotsc]$, which contains the ind-scheme $\Hilb_d(\A^\infty)$ as a subfunctor. We will show that the inclusion $\Hilb_d(\A^\infty)\subset \Hilb_d(\A^\N)$ is an $\A^1$-equivalence, see Proposition~\ref{prop:h^fflat via Hilb}.
\end{rem}

To prove~\eqref{eq:Hilb vs Gr}, we consider the stack $\FFLAT_d$ of finite flat degree $d$ algebras and the stack $\Vect_d$ of rank $d$ locally free sheaves. The forgetful maps
\begin{gather*}
	\Hilb_d(\A^\infty) \to \FFLAT_d,\\
	\Gr_d(\A^\infty) \to \Vect_d
\end{gather*}
are shown to be $\A^1$-equivalences, see Propositions~\ref{prop:h^fflat via Hilb} and~\ref{prop:quot}.
For any ring $R$ and $R$-module $M$ the $R$-module $R\oplus M$ has a commutative unital $R$-algebra structure with $M^2 = 0$, called the \emph{square-zero extension}. This construction defines a morphism \[\alpha\colon\Vect_{d-1}\to \FFLAT_d.\] We show that $\alpha$ is an $\A^1$-homotopy equivalence, with inverse $\beta$ sending an $R$-algebra $A$ to the $R$-module $A/(R\cdot 1_A)$. The $\A^1$-homotopy between $\alpha\beta$ and the identity is obtained using the Rees algebra, making $\FFLAT_d$ a cone over $\Vect_{d-1}$. On the technical level, this is the key point to avoid dealing with the singularities of $\FFLAT_d$.


The stacks $\FFLAT = \coprod_{d} \FFLAT_d$ and $\Vect = \coprod_d \Vect_d$ are commutative monoids under direct sum. While the maps $\alpha$ and $\beta$ do not preserve this structure, we employ them together with McDuff--Segal's group completion theorem to deduce that the forgetful map $\FFLAT\to \Vect$ is an $\A^1$-equivalence after group completion, see Theorem~\ref{thm:main-gp}. As a consequence, we get a new description of algebraic $K$-theory in terms of Hilbert schemes: on affine schemes, there is an $\A^1$-equivalence
\begin{equation}
K \simeq \underline\Z  \times \Hilb_\infty(\A^\infty),
\end{equation}
see Corollary~\ref{cor:Hilb=Gr}.
Here, $K$ is the presheaf of K-theory spaces, $\underline\Z$ is the constant sheaf with value $\Z$, and $\Hilb_\infty(\A^\infty)$ is the colimit of the maps $\Hilb_d(\A^\infty)\to \Hilb_{d+1}(\A^{\infty+1})$.

Furthermore, we discuss consequences of these computations for stable motivic homotopy theory. These results are part of the theory of framed transfers, developed in \cite{voevodsky2001notes}, \cite{garkusha2014framed}, \cite{deloop1} and other works. 
 Some of the main results of this theory are as follows: every generalized motivic cohomology theory acquires a unique structure of framed transfers~\cite[Theorem~3.5.12]{deloop1}, and certain cohomology theories acquire a universality property with respect to their transfers. The big picture of various cohomology theories and the corresponding transfers is given in~\cite[\sectsign 1.1]{hermitian_robbery}. For example, the algebraic cobordism spectrum $\MGL$ represents the universal cohomology theory with finite syntomic transfers, i.e., with pushforwards along finite flat locally complete intersection morphisms. This is expressed by the equivalence of $\infty$-categories
 \[\Mod_{\MGL}(\SH(k)) \simeq \SH^\fsyn(k),\] 
 where $\SH^\fsyn(k)$ is the $\infty$-category of motivic spectra with finite syntomic transfers~\cite[Theorem~4.1.3]{deloop3}.

In Section~\ref{sec:kgl}  we work towards an analogous universal property for algebraic K-theory. We observe that the forgetful map $\FFLAT\to\Vect$ is a morphism of presheaves with framed transfers in the sense of \cite{deloop1}, and the framed suspension spectrum of $\Vect$ is the effective motivic K-theory spectrum $\kgl$. We therefore obtain an equivalence of motivic spectra
\begin{equation*}\label{eq:kgl-fflat}
\Sigma^\infty_{\T,\fr}\FFLAT \simeq \kgl,
\end{equation*}
see Theorem~\ref{thm:kgl-fflat}.
Bachmann employs it in \cite{BachmannFFlatCancellation} to obtain an equivalence of symmetric monoidal $\infty$-categories
\[\Mod_{\kgl}(\SH(k))\left[\tfrac{1}{e}\right] \simeq\SH^\fflat(k)\left[\tfrac{1}{e}\right]\]
where $e$ is the exponential characteristic of $k$. In that sense, $\kgl$ represents the universal generalized motivic cohomology theory with finite locally free transfers. 

Some of our results generalize from $\A^{\infty}_k$
to $\A^{\infty}_X=X\times \A^{\infty}$
for a smooth separated $k$-scheme $X$.
Answering a question by Rahul Pandharipande, we analyze the Hilbert scheme of points of $\A^\infty_X$ and obtain a motivic equivalence
\[\Omega^\infty_\T(\kgl\otimes\Sigma^\infty_\T X_+)\simeq \Hilb(\A^\infty_X)^\gp\]
for $X$ smooth and proper, thus providing a geometric model for the $\kgl$-homology of $X$, see Corollary~\ref{cor: kgl otimes X}. Here, the $\Einfty$-structure on the Hilbert scheme is given by ``taking disjoint unions of closed subschemes''.
As we explain in Section~\ref{sec: A^infty times X}, this equivalence is analogous to Segal's model for $\mathrm{ko}$-homology of a topological space~\cite[Section~1]{segal1977K-homology}. Indeed, the ind-scheme $\Hilb(\A^\infty_X)$ can be thought of as a geometric version of Segal's configuration space of points ``labeled'' by vector spaces: it has a canonical map to the symmetric power $\Sym(X)$, whose fiber over a $k$-rational point of the form $\sum_i d_i[x_i]$ is $\A^1$-equivalent to the stack $\prod_i\Vect_{d_i-1}$.

We proceed with investigating further the connection between algebraic cobordism and algebraic K-theory. In \cite{deloop3}, it is shown that the algebraic cobordism spectrum $\MGL$ is the framed suspension spectrum of the substack $\FSYN\subset\FFLAT$ of finite syntomic schemes.
We show in Section~\ref{sec:MGL->KGL} that the inclusion $\FSYN\subset \FFLAT$ induces the standard orientation map $\MGL\to\kgl$ upon taking framed suspension spectra. Combining the results of this paper with \cite[Theorem 1.1]{deloop4}, we obtain the following commutative diagram in $\Pre(\Sm_k)$, where the symbol $\simeq$ denotes a motivic equivalence and $+$ a motivic equivalence up to Quillen's plus construction:
\[
\begin{tikzcd}
	\Z\times \Hilb_\infty^\lci(\A^\infty) \ar[hook]{r} \ar{d}{\simeq} & \Z\times \Hilb_\infty(\A^\infty) \ar{d}{\simeq}  & \Z \times \Gr_\infty(\A^\infty) \ar{d}{\simeq} \\
	\Z \times \FSYN_\infty \ar[hook]{r} \ar["+"]{d}  & \Z \times \FFLAT_\infty \ar["\simeq"]{r}  \ar["\simeq"]{d} & \Z \times \Vect_\infty \ar["\simeq"]{d} \\
	\FSYN^\gp \ar{r} \ar["\simeq"]{d}  & \FFLAT^\gp \ar["\simeq"]{r}  & \Vect^\gp \ar["\simeq"]{d}  \\
	\Omega^\infty_\T\MGL  \ar{rr} &  &  \Omega^\infty_\T\KGL \rlap.
\end{tikzcd}
\]
Here $\Hilb_\infty(\A^\infty) = \colim_{d,n} \Hilb_d(\A^n)$ is the stabilization of the Hilbert scheme of affine space, and $\Hilb_\infty^\lci(\A^\infty)$ is the local complete intersection locus. 
The bottommost arrow in the diagram is the standard orientation of $\KGL$ and all the other horizontal arrows are the obvious forgetful morphisms. Moreover, the bottom rectangle is one of commutative monoids in presheaves with framed transfers (in particular, of presheaves of $\Einfty$-ring spaces).

One reason for our interest in $\Hilb_d(\A^{\infty})$ is a conjecture of Mike Hopkins stating that the motive of $\Omega^\infty_\T \MGL$ is pure Tate. As observed in \cite[Corollary 1.7]{deloop4}, this conjecture is equivalent to showing that $\Hilb_\infty^\lci(\A^\infty)$ has a pure Tate motive. While this problem remains open, we prove in Section~\ref{sec:deg 3} that $\Hilb_3^\lci(\A^\infty)$ has a pure Tate motive. In contrast, $\Hilb_3^{\lci}(\A^{n})$ is not pure Tate for finite $n$~\cite[Remark~1.8]{deloop4}.

\subsubsection*{Notation and terminology}
$\Sch$ denotes the category of schemes. If $S$ is a scheme, $\Sch_S$ denotes the category of $S$-schemes, $\Sm_S$ that of smooth $S$-schemes, $\H(S)$ the $\infty$-category of motivic spaces over $S$, and $\SH(S)$ that of motivic spectra, see \cite[\sectsign 2.2 and \sectsign 4.1]{norms}.
The \emph{Tate sphere} $\T$ in $\H(S)$
is defined as $\A^1/(\A^1-0)\simeq\Sigma\Gm\simeq (\P^1,1)$.

We only consider the Hilbert scheme of points in this paper. That is,
for a morphism of schemes $X\to S$ and any $S$-scheme $T$,
\[\Hilb_d(X/S)(T)=
\left\{
\begin{matrix}
Z \subset X_T\text{ closed subscheme such that }\\
Z \to T\text{ is finite locally free of degree }d
\end{matrix}
\right\}.
\]
This defines a scheme over $S$
whenever $X$ is affine
or quasi-projective over $S$ \cite[Remark 0B9B]{stacks} (more generally, it is an algebraic space over $S$ when $X$
is separated over $S$ \cite[Theorem~A(i)]{RydhQuot}). We denote by $\Hilb_d^\lci(X/S)$ the open subscheme of $\Hilb_d(X/S)$ given by the locus of local complete intersections. We write $\Hilb(X/S)=\coprod_{d\geq 0}\Hilb_d(X/S)$.  We sometimes omit $S$ from the notation when it is clear from the context.

We denote by $\Pre(\sC)$ the $\infty$-category of presheaves of spaces on $\sC$ and by $\Pre_\Sigma(\sC)\subset \Pre(\sC)$ the full subcategory of presheaves that transform finite sums into finite products \cite[\sectsign 5.5.8]{HTT}. If $\sC$ is a suitable category of schemes, $\Lhtp\colon \Pre(\sC)\to\Pre(\sC)$ (resp.\ $L_\mot\colon \Pre(\sC)\to\Pre(\sC)$) is the localization functor onto the full subcategory of $\A^1$-invariant presheaves (resp.\ of $\A^1$-invariant Nisnevich sheaves). A morphism $f$ in $\Pre(\sC)$ is called an \emph{$\A^1$-equivalence} (resp.\ a \emph{motivic equivalence}) if $\Lhtp(f)$ (resp.\ $L_\mot(f)$) is an equivalence. Thus, if $\sC=\Sm_S$, a motivic equivalence means an isomorphism in the $\infty$-category of motivic spaces $\H(S)$; this should not be confused with an isomorphism in the derived category of motives, $\DM(S)$, which is a weaker condition. Explicitly, a morphism $f\colon \mathcal{X}\to \mathcal{Y}$ in $\Pre(\sC)$ (e.g., a morphism of stacks) is an $\A^1$-equivalence if for every scheme $T$ the morphism $\mathcal{X}(T\times \A^\bullet) \to \mathcal{Y}(T\times\A^\bullet)$ induces an equivalence on geometric realizations.

\subsubsection*{Acknowledgments} We would like to thank Tom Bachmann, Elden Elmanto, Rahul Pandharipande, Bjorn Poonen, Oliver R\"ondigs, Vladimir Sosnilo, Markus Spitzweck, Ravi Vakil, Andrzej Weber, and the referees for discussions and suggestions. This work was partly initiated during Yakerson's visit to the Institute of Mathematics of Polish Academy of Sciences, and then continued during Jelisiejew's visit to the University of Regensburg. We would like to express our gratitude to both departments for their hospitality. We also thank the organizers of the Southern California Algebraic Geometry Seminar and the ETH Z\"urich Algebraic Geometry and Moduli Zoominar.

\section{The stacks $\FFLAT_d$ and $\Vect_{d-1}$ are
$\A^1$-equivalent}\label{sec:FFLATtoVect}

We say that a morphism of schemes $p\colon Y\to X$ is \emph{finite locally free} if it is finite, flat, and of finite presentation, or equivalently if it is affine and $p_*\sO_Y$ is a locally free $\sO_X$-module of finite rank.
We denote by $\Vect(X)$ the groupoid of finite locally free $\sO_X$-modules and by $\FFLAT(X)$ the groupoid of finite locally free $X$-schemes, or equivalently of finite locally free $\sO_X$-algebras. For $d\geq 0$, write $\FFLAT_d\subset \FFLAT$ for the subfunctor of finite locally free commutative $\sO_X$-algebras of rank $d$, and $\Vect_d\subset\Vect$ for the subfunctor of vector bundles of rank $d$, which is the classifying stack of the group scheme $\GL_d$.

The presheaves of groupoids $\Vect_d$ and $\FFLAT_d$ are algebraic stacks of finite type over $\Z$. In particular, Poonen observed that $\FFLAT_d$ is the quotient stack by $\GL_d$ of the affine scheme $Y_d$ of commutative algebra structures on $\A^d$, with $\A^d$ viewed as a free module of rank $d$ over $\Z$ \cite[Definition 3.2]{Poonen}. The singularities of the stack $\FFLAT_d$ are known to be almost arbitrarily bad: namely, $\coprod_{d\geq 0} \FFLAT_d$ satisfies Murphy's law up to retraction \cite[after Theorem 1.3]{JelisiejewPathologies}.

For $d\geq 1$,
we define morphisms $\alpha\colon \Vect_{d-1}\to \FFLAT_d$
and $\beta\colon \FFLAT_d\to \Vect_{d-1}$ as follows.
Given a vector bundle $\sV$ of rank $d-1$
over a scheme $X$,
let $\alpha(\sV)$ be the sheaf of $\sO_X$-algebras
$\sO_X\oplus \sV$, with $\sV$
a square-zero ideal. Given a sheaf of algebras
$\sA$ which is locally free of rank $d$ over $X$, let $\beta(\sA)$
be $\sA/\sO_X$, which is a vector bundle of rank $d-1$.
(Indeed, the global section $1$ of $\sA$ is nonzero at every point of $X$, so
locally we can complete it to a basis of $\sA$ and obtain that
$\sO_X \to \sA\to \sA/\sO_X$ is locally split exact.)

\begin{thm}\label{thm:stack-equivalence}
For each $d\geq 1$, the morphism $\alpha\colon \Vect_{d-1}\to\FFLAT_d$ in $\Pre(\Sch)$
is an $\A^1$-equivalence, with inverse $\beta$ up to $\A^1$-homotopy.
\end{thm}

The proof uses a degeneration of commutative algebras
which we formulate in terms of the Rees algebra construction
(Lemma \ref{lem:rees}).
Poonen considered a closely related degeneration
\cite[Proposition 7.1]{Poonen}.

The following is standard when $R$ is a field,
a recent reference being \cite[Remark 2.6.5]{Schedler}.

\begin{lem}\label{lem:rees}
Let $R$ be a commutative ring, $A$ an associative $R$-algebra with
an increasing filtration by $R$-submodules
$A_0\subset A_1\subset\cdots$ with
$1\in A_0$, $A_i\cdot A_j\subset A_{i+j}$,
and $A=\bigcup_i A_i$. Assume that the $R$-modules
$\gr_iA=A_i/A_{i-1}$ are flat for all $i\geq 0$ (where $A_{-1}=0$). Let
\[
\Rees(A)=\bigoplus_{i\geq 0}A_i t^i\subset A[t]
\]
be the corresponding Rees algebra. Then $\Rees(A)$ is a flat $R[t]$-algebra with
$\Rees(A)/(t)\simeq \gr_*(A)$ and $\Rees(A)/(t-1)\simeq A$.
If $\gr_iA$ is finite locally free as a $R$-module for each $i$
and zero for $i$ sufficiently large, then $\Rees(A)$ is finite locally
free as an $R[t]$-module.
\end{lem}

\begin{proof}
It is immediate that $\Rees(A)/(t)\simeq \gr_*(A)$
and $\Rees(A)/(t-1)\simeq A$. The conclusions about flatness
are also straightforward; we only write out the proof
in the case needed below,
where $\gr_iA$ is finite locally free as an $R$-module
for each $i$ and zero for $i$ sufficiently large.
In this case, $\gr_i(A)$ is projective as an $R$-module for each $i$
\cite[Lemma 00NX]{stacks}, and so each inclusion $A_{i-1}\to A_i$
has an $R$-linear splitting. These splittings determine
an isomorphism $\Rees(A)\simeq \gr_*(A)[t]$
of $R[t]$-modules. Here $\gr_*(A)$ is a finite locally free
$R$-module, and so $\Rees(A)$ is a finite locally free
$R[t]$-module.
\end{proof}

\begin{proof}[Proof of Theorem~\ref{thm:stack-equivalence}]
The composition $\Vect_{d-1}\to\FFLAT_d\to\Vect_{d-1}$ is isomorphic
to the identity. It remains to show that the composition
$\FFLAT_d\to\Vect_{d-1}\to \FFLAT_d$ is $\A^1$-homotopic
to the identity. We define an explicit $\A^1$-homotopy as follows.

Let $R$ be a ring and let $A$ be a finite locally free $R$-algebra
of rank $d$. The point is that $A$ has a canonical increasing filtration
given by $A_0=R\cdot 1_A$
and $A_1=A$. The $R$-modules
$\gr_i(A)$ are locally free, and so Lemma \ref{lem:rees} applies.
The associated Rees algebra $\Rees(A)$ is locally free
of rank $d$ over $R[t]$, with $\Rees(A)/(t)\simeq \gr_*(A)$
and $\Rees(A)/(t-1)\simeq A$. Here $\gr_*(A)$ is
the square-zero extension $R\oplus A/(R\cdot 1)$.

The operation $A\mapsto \Rees(A)$ is functorial in $A$
and hence defines a morphism of stacks over $\Z$,
\[
\A^1\times \FFLAT_d\to \FFLAT_d.
\]
At $t=1$ in $\A^1$, this is the identity, while at $t=0$
it is isomorphic to the composition $\FFLAT_d\to\Vect_{d-1}\to \FFLAT_d$, sending $\sA$ to the square-zero extension $\sO_X\oplus \sA/(\sO_X\cdot 1_{\sA})$. This is the desired $\A^1$-homotopy.
\end{proof}

\begin{rem}\label{rem:stack-equivalence}
	Theorem~\ref{thm:stack-equivalence} implies that $\alpha\colon \Vect\to \FFLAT_{\geq 1}$ is an $\A^1$-equivalence in $\Pre(\Sch)$, with inverse $\beta$ up to $\A^1$-homotopy, since sheafification preserves $\A^1$-homotopies.
\end{rem}

\section{$\A^1$-equivalence between the group completions of the stacks $\FFLAT$ and $\Vect$}

The direct sum and tensor product of $\sO_X$-modules define an $\Einfty$-semiring structure on the stack $\Vect$.
Similarly, the disjoint union and cartesian product of schemes define an $\Einfty$-semiring structure on $\FFLAT$. In order to reduce the technical burden here, we postpone the formal construction of these semirings to Section~\ref{sec:kgl}. We denote by $\Vect^\gp$ and $\FFLAT^\gp$ the corresponding group completions, which are presheaves of $\Einfty$-ring spaces on the category of schemes. 

Forgetting the algebra structure on a finite locally free $\sO_X$-algebra defines a morphism of $\Einfty$-semirings $\eta\colon \FFLAT\to\Vect$. The main result of this section is the following theorem:

\begin{thm}\label{thm:main-gp}
	The map $\eta^\gp\colon \FFLAT^\gp \to \Vect^\gp$ is an $\A^1$-equivalence.
\end{thm}

Let $\sigma\colon \FFLAT \to \FFLAT$ be the map that adds a disjoint point over the base, and let $\tau\colon \Vect\to\Vect$ be the map that adds a trivial line bundle.
We denote by $\FFLAT^\st$ the colimit of the sequence
\[
\FFLAT\xrightarrow{\sigma} \FFLAT \xrightarrow{\sigma}\FFLAT\to\dotsb,
\]
 We similarly define $\Vect^\st$ as the colimit of the sequence
 \[
 \Vect\xrightarrow{\tau} \Vect \xrightarrow{\tau}\Vect\to\dotsb.
 \]
There are canonical maps $\FFLAT^\st\to\FFLAT^\gp$ and $\Vect^\st\to\Vect^\gp$. 

Note that there is a commutative square
\[
\begin{tikzcd}
	\FFLAT \ar{r}{\eta} \ar{d}{\sigma} & \Vect \ar{d}{\tau} \\
	\FFLAT \ar{r}{\eta} & \Vect\rlap,
\end{tikzcd}
\]
inducing a map $\eta^\st\colon \FFLAT^\st\to \Vect^\st$ in the colimit.
We shall deduce Theorem~\ref{thm:main-gp} from the following variant, which does not involve group completion:

\begin{thm}\label{thm:main-st}
	The map $\eta^\st\colon \FFLAT^\st\to \Vect^\st$ is an $\A^1$-equivalence.
\end{thm}

To prove Theorem~\ref{thm:main-st}, we consider the algebraic stack $\FFmrk$ of finite locally free schemes with a distinguished
point. More precisely, for every scheme $X$, the groupoid $\FFmrk(X)$ is given by
\[
    \FFmrk(X) = \left\{ (f, s)\ |\ f\colon
    Z\to X \mbox{ finite locally free, } s\colon X\to Z\mbox{ a
    section of } f\right\}.
\]
We will refer to the section $s$ as the \emph{marking} and to such pairs $(f,s)$ as
\emph{marked} schemes.
There is a forgetful map $\theta \colon \FFmrk \to \FFLAT$ that discards the marking (which is the universal finite locally free family).

Let $\FFnu$ be the stack of finite locally free sheaves of nonunital commutative algebras.
There is an equivalence
\[
\FFnu \simeq \FFmrk
\]
sending a nonunital algebra $\sA$ to its unitalization $\sO\oplus\sA$, with inverse sending an augmented algebra to its augmentation ideal. 
Let $\nu\colon \Vect\to\FFnu$ be the functor sending a finite locally free sheaf $\sV$ to $\sV$ regarded as a nonunital algebra with zero multiplication.

The following table summarizes the various maps we will use in the proof of Theorem~\ref{thm:main-st}, some of which are introduced below:

\begin{center}
  \begin{tabular}{l l}
  name & description\\
    \toprule
	 $\eta\colon \FFLAT \to \Vect$ & forgets the algebra structure\\
	 $\theta\colon \FFmrk \to \FFLAT$ & forgets the marking\\
        $\tau\colon\Vect \to \Vect$ & adds a trivial line bundle\\
    $\sigma\colon\FFLAT \to \FFLAT$ & adds a disjoint point\\
$\sigma^\mrk\colon \FFmrk\to \FFmrk$ & adds a disjoint point without changing the marking\\
$\varepsilon\colon\FFmrk \to \FFmrk$ & adds a new tangent direction at the
    marked point \\
      $\nu\colon\Vect \to \FFLAT^\mathrm{nu}$ & equips a module with zero multiplication\\
		$\alpha\colon \Vect \to \FFLAT$ & forms the trivial square-zero extension\\
    $\pi\colon \FFLAT^\mathrm{nu} \to \Vect$ & forgets the algebra structure
\end{tabular}
\end{center}
\vskip\parskip

\begin{prop}\label{prop:vect-nu}
	The map $\nu\colon\Vect\to \FFnu$ is an $\A^1$-equivalence.
\end{prop}

\begin{proof}
	Let $\pi\colon \FFnu\to \Vect$ be the forgetful functor. Then $\pi\circ\nu$ is the identity, and the map
	\[
	\FFnu \to \Maps(\A^1,\FFnu), \quad \sA \mapsto t\sA[t]
	\]
	is an $\A^1$-homotopy from $\nu\circ\pi$ to the identity of $\FFnu$.
\end{proof}

\begin{prop}\label{prop:bank-robbery}
	The map $\theta\colon \FFLAT^{\mrk}\to \FFLAT_{\geq 1}$ is an $\A^1$-equivalence.
\end{prop}

\begin{proof}
	We have a commutative square
	\[
	\begin{tikzcd}
	    \Vect \ar["\alpha"]{r} \ar[swap,"\nu"]{d} & \FFLAT_{\geq 1}  \\
		\FFnu \ar["\simeq"]{r} & \FFLAT^\mrk \rlap. \ar[swap,"\theta"]{u}
	\end{tikzcd}
	\]
	The map $\alpha$ is an $\A^1$-equivalence by Remark~\ref{rem:stack-equivalence} and the map $\nu$ is an $\A^1$-equivalence by Proposition~\ref{prop:vect-nu}. By the 2-out-of-3 property, $\theta$ is an $\A^1$-equivalence.
\end{proof}

Let $\sigma^\mrk\colon \FFmrk\to \FFmrk$ be the map that adds a
disjoint point over the base without changing the marking. 
We then have a commutative square
\[
\begin{tikzcd}
    \FFmrk \ar["\theta"]{r} \ar["\sigma^\mrk"]{d} & \FFLAT \ar["\sigma"]{d}  \\
	\FFmrk \ar["\theta"]{r} & \FFLAT\rlap,
\end{tikzcd}
\]
inducing in the colimit a map
\[
\theta^\st\colon \FFLAT^{\mrk,\st}\to \FFLAT^\st.
\]

\begin{cor}\label{cor:theta-st}
	The map $\theta^\st\colon \FFLAT^{\mrk,\st}\to \FFLAT^\st$ is an $\A^1$-equivalence.
\end{cor}

\begin{proof}
	This follows immediately from Proposition~\ref{prop:bank-robbery} and the observation that the inclusion $\FFLAT_{\geq 1} \subset \FFLAT$ becomes an equivalence after stabilization along $\sigma$.
\end{proof}

\begin{lem}\label{lem:fflat-pushout}
	Let $S$ be a scheme and let $Z\to X$ and $Z\to Y$ be morphisms between finite locally free $S$-schemes with $Z\to X$ a closed immersion. Then the pushout $X\sqcup_ZY$ in the category of schemes exists and is finite locally free over $S$.
\end{lem}

\begin{proof}
    The pushout exists and is affine over $S$ by~\cite[Tag 0E25]{stacks}. Now,
    local freeness translates into an algebraic statement: if $M\to P$ and
    $N\to P$ are morphisms of finite locally free modules with $M\to P$
    surjective, then $M\times_PN$ is finite locally free. This is true, since
    the module
    $M\times_P N$ is an extension of the locally free module $\ker(M\to P)$ by
    the locally free module $N$.
\end{proof}

\begin{proof}[Proof of Theorem~\ref{thm:main-st}]
	Consider the finite locally free $\Z[t]$-scheme $\robber = \Spec \Z[x,
	t]/((x-t)x)$ whose fibers over $0$ and $1$ are the schemes $\robbernot = \Spec\Z[x]/(x^2)$ and $\robber_1 = \Spec\Z\sqcup\Spec \Z$.
	There is a section $i\colon \Spec \Z[t] \to \robber$ sending $x$ to $0$,
	which defines an element $(\robber, i)$ in
	$\FFmrk(\Z[t])$.

	Let $\varepsilon \colon \FFmrk \to \FFmrk$ be the map that sends
	$(f\colon Z\to S, s\colon S\to Z)$ to $Z\sqcup_{S} (\robbernot)_S$ with
	section the canonical map $S\to Z \sqcup_S (\robbernot)_S$. The map $\varepsilon$ is well-defined by Lemma~\ref{lem:fflat-pushout}. In plain words, the transformation $\varepsilon$ adds a new tangent direction at the
	marked point of the finite scheme.
	Similarly, let $H\colon \FFmrk\to \Maps(\A^1, \FFmrk)$ be the map that
	sends $(f\colon Z\to S, s\colon S\to Z)$ to the pushout $\A^1_Z\sqcup_{\A^1_S} \robber_S$.
	It is then clear that $H_0=\varepsilon$ and $H_1 = \sigma^\mrk$, i.e., $H$ is an $\A^1$-homotopy from $\varepsilon$ to $\sigma^\mrk$.
	
	We now have commutative squares
	\[
	\begin{tikzcd}
	    \Vect \ar["\nu"]{r} \ar["\tau"]{d} & \FFnu \simeq \FFmrk \ar["\varepsilon",d,shift left=2em] & \FFmrk \ar["\theta"]{r} \ar["\sigma^\mrk"]{d} & \FFLAT 
	    \ar["\sigma"]{d}  \ar["\eta"]{r} & \Vect \ar["\tau"]{d} \\
		\Vect \ar["\nu"]{r} & \FFnu \simeq \FFmrk & \FFmrk \ar["\theta"]{r} & \FFLAT \ar["\eta"]{r} & \Vect
	\end{tikzcd}
	\]
	with the following properties:
	\begin{enumerate}
		\item $\nu\colon \Vect\to \FFnu$ is an $\A^1$-equivalence, by Proposition~\ref{prop:vect-nu};
		\item there is an $\A^1$-homotopy $H\colon \varepsilon \rightsquigarrow \sigma^\mrk$;
		\item $\theta^\st\colon \FFLAT^{\mrk,\st} \to \FFLAT^\st$ is an $\A^1$-equivalence, by Corollary~\ref{cor:theta-st};
		\item the $\A^1$-homotopy $\eta\circ \theta\circ H\circ \nu$ is constant if we identify its endpoints using the isomorphism
		\[
		\eta\circ \theta\circ\nu\circ\tau
		\simeq
		\tau\circ\tau\stackrel{\mathrm{swap}}\simeq\tau\circ\tau
		\simeq 
		\tau\circ \eta\circ \theta\circ \nu.
		\]
	\end{enumerate}
	Assertion (4) is a straightforward verification from the definition of the $\A^1$-homotopy $H$.
	Assertion (2) allows us to define a map $\nu^\st\colon \Lhtp\Vect^\st \to \Lhtp\FFLAT^{\mrk,\st}$, which is an equivalence by (1). Assertion (4) implies that $\Lhtp(\eta^\st\circ\theta^\st)\circ\nu^\st\simeq \Lhtp\tau^\st$, where $\tau^\st\colon \Vect^\st\to\Vect^\st$ is the action of $\sO\in\Vect$ on the $\Vect$-module $\Vect^\st$. Note that $\tau^\st$ is \emph{not} an equivalence, but it is an $\A^1$-equivalence by \cite[Proposition 5.1]{deloop4} since the cyclic permutation of $\sO^3$ is $\A^1$-homotopic to the identity. From this and (3) we deduce that $\eta^\st$ is an $\A^1$-equivalence.
\end{proof}

\begin{proof}[Proof of Theorem~\ref{thm:main-gp}]
	Let $\FFLAT[-1]$ and $\Vect[-1]$ be the commutative monoids obtained from $\FFLAT$ and $\Vect$ by inverting $\sO$. We have a commutative square
	\[
	\begin{tikzcd}
		\FFLAT^\st \ar{r} \ar{d} & \FFLAT[-1] \ar{d} \\
		\Vect^\st \ar{r} & \Vect[-1]\rlap.
	\end{tikzcd}
	\]
	The cyclic permutation of $\sO^{\oplus 3}$ becomes homotopic to the identity in $\Lhtp\Vect$, since it is a product of elementary matrices. It thus follows from \cite[Proposition 5.1]{deloop4} that the lower horizontal map is an $\A^1$-equivalence.
	By Theorem~\ref{thm:main-st}, the left vertical map is an $\A^1$-equivalence. It then follows from \cite[Proposition 5.1]{deloop4} that the upper horizontal map is also an $\A^1$-equivalence. Hence, the right vertical map is an $\A^1$-equivalence.
	 Since the functor $\Lhtp$ commutes with group completion \cite[Lemma 5.5]{HoyoisCdh}, we deduce that $\FFLAT^\gp\to\Vect^\gp$ is an $\A^1$-equivalence.
\end{proof}
 
 \begin{rem}
 	It follows from Theorem~\ref{thm:main-st} and \cite[Proposition 5.1]{deloop4} that the cyclic permutation of three points in $\FFLAT$ becomes $\A^1$-homotopic to the identity in $\FFLAT^\st$. In fact, one can directly show that it is $\A^1$-homotopic to the identity in $\FFLAT$ as follows. Let $X=\Z[x,y,z]/(xy,xz,yz)$ be three affine lines glued together at the origin. There is a finite flat morphism $X\to\A^1=\Spec \Z[t]$ given by $t\mapsto x+y+z$, which is an $\A^1$-homotopy between three points and the finite flat $\Z$-scheme $X_0=\Spec \Z[x,y,z]/(xy,xz,yz,x+y+z)$. The cyclic group $C_3$ acts on $X/\A^1$ by permuting the three lines, so it remains to show that the induced automorphism $c$ of $X_0$ is $\A^1$-homotopic to the identity. 
Note that $X_0$ is a square-zero extension $\Spec \Z\oplus V$ where $V=\Z\{x,y,z\}/\Z\{x+y+z\}$. Hence the group $\GL(V)$ acts on $X_0$, and the matrix of $c$ in the basis $(x,y)$ of $V$ is
\(\left(
\begin{smallmatrix}
	0 & -1 \\
	1 & -1
\end{smallmatrix}
\right)\),
which has determinant $1$. We conclude using the fact that $(\Lhtp\SL_2)(\Z)$ is connected, since $\SL_2(\Z)$ is generated by elementary matrices of row additions.
 \end{rem}

\section{Consequences for the Hilbert scheme of $\A^\infty$} \label{sec: A^infty times X}

Fix a base scheme $S$. For a separated $S$-scheme $X$ we write $\Hilb(\A^\infty_X/S)$ for $\colim_n \Hilb(\A^n_X/S)$, where the colimit is taken in $\Pre(\Sch_S)$. 
For a set $E$, we also consider $\Hilb(\A^E_X/S)$, the Hilbert scheme of points of the scheme $\A^E_X = \Spec \Z[x_e\ |\ e\in E] \times X$ over $S$.

    Let us recall a useful lemma from~\cite{deloop1}. A presheaf $\mathcal{F}\colon\Sch_S^\op\to \Spc$ \emph{satisfies closed gluing}~\cite[Definition~A.2.1]{deloop1} if it sends the empty scheme to the point and it sends pushouts of closed embeddings to pullbacks. Presheaves parametrizing embeddings often satisfy closed gluing.
	 
	 We shall say that a morphism $f$ in $\Pre(\Sch_S)$ is a \emph{universal $\A^1$-equivalence on affine schemes} if any base change of $f$ in $\Pre(\Sch_S)$ is an $\A^1$-equivalence on affine schemes.

    \begin{lem}\label{lem:deloop1_criterion_for_A1equiv}
		\leavevmode
     \begin{enumerate}
     \item Let $\mathcal{F}\colon\Sch_S^\op\to \Spc$ be a presheaf of spaces that satisfies closed gluing and assume that for every affine scheme $\Spec A$ over $S$ and finitely generated ideal $I\subset A$ the morphism
     \[\mathcal{F}(\Spec A)\to \mathcal{F}(\Spec A/I)\]
     is surjective on $\pi_0$. Then $\mathcal{F}$ is $\A^1$-contractible on affine schemes.
     
     \item Let $f\colon\mathcal{F}\to \mathcal{G}$ be a map of presheaves of spaces on $\Sch_S$ such that for every $T\in\Sch_S$ the fiber over every $T$-point $T\to \mathcal{G}$ satisfies closed gluing (as a presheaf on $\Sch_T$). Suppose that for every affine scheme $\Spec A$ over $S$ and every finitely generated ideal $I\subset A$ the morphism
     \[\mathcal{F}(\Spec A)\to \mathcal{F}(\Spec A/I)\times_{\mathcal{G}(\Spec A/I)}\mathcal{G}(\Spec A)\]
     is surjective on $\pi_0$. Then $f$ is a universal $\A^1$-equivalence on affine schemes.
     \end{enumerate}
    \end{lem}
	 
    \begin{proof}
     We will use the same strategy as in the proof of~\cite[Lemma~2.3.22]{deloop1}. Recall that $(\Lhtp\mathcal{F})(\Spec A)$ is the geometric realization of the simplicial space
     \[\textstyle[n]\mapsto \mathcal{F}(\Delta^n_A)=\mathcal{F}\left(\Spec A[t_0,\dots,t_n]/\left(\sum_i t_i - 1\right)\right).\]
     Since $\mathcal{F}$ satisfies closed gluing, \cite[Lemma~A.2.6]{deloop1} implies that
     \[\Maps(\partial \Delta^n,\mathcal{F}(\Delta^\bullet_A))\simeq \mathcal{F}(\partial\Delta^n_A)\]
     where $\partial\Delta^n_A$ is the zero locus of $t_0\cdots t_n$ in $\Delta^n_A$. In particular the map
     \[\Maps(\Delta^n,\mathcal{F}(\Delta^\bullet_A))\to \Maps(\partial\Delta^n,\mathcal{F}(\Delta^\bullet_A))\]
     is surjective on $\pi_0$. Therefore the geometric realization of $\mathcal{F}(\Delta^\bullet_A)$ is contractible \cite[Lemma A.5.3.7]{SAG}.
     
     The second statement follows from the first using the universality of colimits.
    \end{proof}

    For an $S$-scheme $X$, let $\hfflat_S(X)$ (resp.\ $h^\mathrm{fsyn}_S(X)$) be the presheaf of groupoids that sends an $S$-scheme $U$ to the groupoid of spans $U \leftarrow Z \rightarrow X$
    where $Z \to U$ is finite locally free (resp.\ finite syntomic) and the morphisms are over $S$.
    In particular, $\hfflat_S(S)$ is the presheaf $\FFLAT$ restricted to $S$-schemes. There is a forgetful map $\Hilb(\A^\infty_X/S)\to \hfflat_S(X)$ sending a closed subscheme $Z\hook \A^\infty_X \times_S U$ to the span $U\leftarrow Z\rightarrow X$.

    \begin{prop} \label{prop:h^fflat via Hilb}
        Let $X$ be a separated $S$-scheme.
        \begin{enumerate}
            \item\label{it:potencialInfinity} The forgetful map $\Hilb(\A^\infty_X/S)\to \hfflat_S(X)$ is a universal $\A^1$-equivalence on affine schemes.
            \item\label{it:actualInfinity} If $E$ is an infinite set and $X\to S$
                is affine, the forgetful map
                $\Hilb(\A^E_X/S)\to \hfflat_S(X)$ is a universal $\A^1$-equivalence on affine schemes.
        \end{enumerate}
    \end{prop}
	 
	 \begin{proof}
		  The presheaves $\Hilb(\A^\infty_X/S)$, $\Hilb(\A^E_X/S)$, and $\hfflat_S(X)$ all satisfy closed gluing, so Lemma~\ref{lem:deloop1_criterion_for_A1equiv}(2) applies. 
		  Suppose $U=\Spec A$ is an affine scheme over $S$ and let $U_0=\Spec A/I$ for some finitely generated ideal $I=(i_1,\dots,i_r)$. Let $U\leftarrow Z=\Spec B\to X$ be a span in $\hfflat_S(X)(U)$ and let $Z_0=Z\times_UU_0$.
		  To prove \eqref{it:potencialInfinity}, we must show that every closed embedding
		  $Z_0 \hook \A^r_{X\times_S U_0}$ over $X\times_S U_0$ can be extended to a closed embedding $Z\hook \A^{r+m}_{X\times_S U}$ over $X\times_S U$, for some $m\geq 0$:
        \[
            \begin{tikzcd}
                Z_0\arrow[d, hook]\arrow[r, hook] & \A^r_{X \times_S U_0} \arrow[r, "\operatorname{inj}", hook] & \A^{r+m}_{X \times_S U_0}\arrow[d, hook]\\
                Z \arrow[rr, dashed, hook]\arrow[dr] && \A^{r+m}_{X \times_S U}\arrow[ld]\\
                & X\times_S U
            \end{tikzcd}
        \]
        Here, $\operatorname{inj}$ is the canonical map adding zero coordinates.
        
        We construct the embedding $Z\to \A^{r+m}_{X \times_S U} = \A^{r+m}\times X \times_S U$ as follows. The map to $X\times_S U$ is already given, and we lift the map $Z_0\to \A^r$ to any map $Z\to \A^r$. It remains to fix $m$ and the map $Z\to \A^m$. Such a map corresponds to a tuple of global functions on $Z$.
        
        Let $i_1, \ldots , i_{r}$ be generators of the ideal $I\subset A$ and let $b_1, \ldots , b_e$ be generators of $A$-module $B$. Let $m=re$ and take the $m$-tuple of functions given by all possible products
        $\{i_{r'}b_{e'}\}_{r'\leq r, e'\leq e}$. By assumption, the scheme $Z$ is finite over $U$ hence over $\A^{r+m}_U$. Since $X$ is separated over $S$, the scheme $Z$ is finite over
        $\A^{r+m}_{X \times_S U}$ as well~\cite[Remark 9.11]{WedhornAG}.
        Therefore to prove that $f\colon Z\to \A^{r+m}_{X \times_S U}$ is a closed immersion it remains to show that $f^\sharp:\sO_{\A^{r+m}_{X\times_S U}} \to f_*(\sO_Z)$ is a surjection of sheaves. Both sheaves are $\sO_{\A^{r+m}_{X\times_S U}}$-modules of finite type, and so we can use Nakayama's lemma and check the condition on the fibers over the image of any point of $Z$. The statement is obvious for the points of $Z_0$, since we know that $Z_0\to \A^{r+m}_{X\times_S U_0}$ is a closed immersion. Every point $z\in Z-Z_0$ has a distinguished open neighborhood in $Z$ where one of the generators of $I$, say $i_s$, is invertible. But then the image of $f^\sharp$ contains $i_sb_1,\dots,i_sb_e$ which generate $\sO_{Z,z}$ as an $\sO_{\A^{r+m}_{X \times_S U},f(z)}$-module. Therefore $f^\sharp$ is surjective at $z$ as well. This concludes the proof of~\eqref{it:potencialInfinity}.
        
			We now prove~\eqref{it:actualInfinity}. By the assumption, $X \times_S U$ is affine; denote it by $\Spec C$.
         We must then prove that every generating family $(\bar b_e)_{e\in E}$ of the $C$-algebra $B/IB$ can be lifted to a generating family $(b_e)_{e\in E}$ of $B$. Since $B/IB$ is a finitely generated $C$-algebra, there exists a finite subset $E'\subset E$ such that $(\bar b_e)_{e\in E'}$ generates $B/IB$; to construct it fix a set of finitely many generators of $B/IB$ over $C$ and then pick $E'$ indexing the elements of $(\bar b_e)_{e\in E}$ that appear in expressions for these generators. Choose arbitrary lifts $b_e\in B$ of $\bar b_e$ for $e\in E'$.

         Since $B$ is a finite $C$-algebra and $I$ is finitely generated, $IB$ is a finite $C$-module. Let $h_1,\dotsc,h_m$ be generators of $IB$ as a $C$-module. Note that the elements $h_i$ and $b_e$ for $e\in E'$ generate $B$ as a $C$-algebra.
			Since $E$ is infinite, there exist distinct elements $e_i\in E-E'$ for $1\leq i\leq m$. For every such $i$, choose an element $r_i$ lifting $\bar b_{e_i}$ in the $C$-subalgebra of $B$ generated by $b_{e}$ for $e\in E'$. Then $b_{e_i}:=r_i+h_i$ is also a lift of $\bar b_{e_i}$, and the elements $b_{e_i}$ and $b_e$ for $e\in E'$ generate $B$ as a $C$-algebra.
         Choosing the remaining lifts $b_e$ arbitrarily, we thus obtain a generating family $(b_e)_{e\in E}$ as desired.
    \end{proof}
	
    \begin{rem}
        The separatedness assumption in Proposition~\ref{prop:h^fflat via
        Hilb} is necessary for the following reason. Let $X$ be an affine line
        with doubled origin and $i\colon \A^1\hook X$ be one of the
        inclusions. View $\A^1$ as a degree one family over itself. Then there does not exist a closed immersion $\A^1\hook\A^{m+1}_{X}$ whose
        projection is $i$.
    \end{rem}
	 
    It follows from Proposition~\ref{prop:h^fflat via Hilb} that for every affine $S$-scheme $X$ the inclusion $\Hilb_d(\A^\infty_X/S)\subset \Hilb_d(\A^\N_X/S)$ is an $\A^1$-equivalence on affine schemes, so we concentrate on the former space.
	 
\begin{cor}\label{cor:Hilb=FFlat}
	The forgetful map $\Hilb(\A^\infty)\to \FFLAT$ is a universal $\A^1$-equivalence on affine schemes.
\end{cor}

\begin{proof}
	Take $X=S=\Spec\Z$ in Proposition~\ref{prop:h^fflat via Hilb}.
\end{proof}

\begin{cor}\label{cor:Hilb=Gr}
	The forgetful map
	\[
	\underline\Z \times \Hilb_\infty(\A^\infty) \to K
	\]
	is an $\A^1$-equivalence on affine schemes, where $K$ is the presheaf of $K$-theory spaces.
	In particular, if $R$ is a regular noetherian commutative ring, then
	\[
	K(R) \simeq \Z^{\Spec R} \times \left\lvert \Hilb_\infty(\A^\infty)(\Delta^\bullet_R)\right\rvert.
	\]
\end{cor}

\begin{proof}
	This map factors as
	\[
	\underline\Z \times \Hilb_\infty(\A^\infty) \to \FFLAT^\st \to \Vect^\st \to K.
	\]
	The first map is an $\A^1$-equivalence on affine schemes by Corollary~\ref{cor:Hilb=FFlat}, and the second map is an $\A^1$-equivalence by Theorem~\ref{thm:main-st}. The third map is also an $\A^1$-equivalence on affine schemes, see for example \cite[Example 5.2]{deloop4}. The last statement follows because $K\simeq\Lhtp K$ on regular noetherian schemes.
\end{proof}

\begin{cor}
	Let $X$ be a separated $S$-scheme. Then the forgetful map $\Hilb^\lci(\A^\infty_X/S)\to h_S^\mathrm{fsyn}(X)$ is a universal $\A^1$-equivalence on affine schemes.
\end{cor}

\begin{proof}
    We have a tautological pullback square
    \[
        \begin{tikzcd}
            \Hilb^\lci(\A^\infty_X/S)\ar[r]\ar[d] & h_S^\mathrm{fsyn}(X)\ar[d]\\
            \Hilb(\A^\infty_X/S)\ar[r] & \hfflat_S(X)\rlap,
        \end{tikzcd}
    \]
	so the claim follows from Proposition~\ref{prop:h^fflat via Hilb}\eqref{it:potencialInfinity}.
\end{proof}

    It is well known that the forgetful map $\Gr_d(\A^\infty)\to\Vect_d$ is an $\A^1$-equivalence on affine schemes. We give a proof in Proposition~\ref{prop:quot} below, since we could not locate a reference. In fact, we consider the following slightly more general situation, which is the linearized version of Proposition~\ref{prop:h^fflat via Hilb}. For $X\to S$ an affine morphism of schemes and $\sF$ a quasi-coherent sheaf on $X$, the scheme $\Quot_d(X/S,\sF)$ parametrizes quotients of $\sF$ that are finite locally free of rank $d$ over $S$, see~\cite[Theorem~4.7]{Gustavsen_Laksov_Skjelnes__Quot}. We denote by $\COH_d(X/S)$ the presheaf of groupoids on $S$-schemes parametrizing quasi-coherent sheaves on $X$ that are finite locally free of rank $d$ over $S$ (which is an algebraic stack at least when $X\to S$ is of finite type, see~\cite[\sectsign 4]{RydhQuot}). There is a forgetful map
$\Quot_d(X/S,\sF) \to \COH_d(X/S)$ that for $X=S$ becomes $\Gr_d(\sF)\to\Vect_d$. We define
\[
\Quot_d(X/S,\sO_X^\infty) = \colim_n\Quot_d(X/S,\sO_X^n).
\]

\begin{prop}\label{prop:quot}
	Let $S$ be a scheme, $X$ an affine $S$-scheme, and $d\geq 0$. Then the forgetful map
	\[
	\Quot_d(X/S,\sO_X^\infty) \to \COH_d(X/S)
	\]
	is a universal $\A^1$-equivalence on affine schemes.
\end{prop}

\begin{proof}
    By Lemma~\ref{lem:deloop1_criterion_for_A1equiv}(2), it suffices to show that for any affine scheme $Y=\Spec A$ and any finitely presented closed subscheme  $Y_0=\Spec A/I$, the map
 \[\Quot_d(X/S,\mathcal{O}_X^{\infty})(Y) \to \Quot_d(X/S,\mathcal{O}_X^{\infty})(Y_0)\times_{\COH_d(X/S)(Y_0)} \COH_d(X/S)(Y)\]
 is surjective.
 Concretely, if $Y\times_SX=\Spec B$, we must show that for any $B$-module $M$ that is finite locally free of rank $d$ over $A$ and any generators $(\bar x_1,\dots, \bar x_r)$ of $M/IM$ as a $B/IB$-module, we can find generators $(x_1,\dots,x_{r+s})$ of $M$ as a $B$-module such that $x_1,\dots,x_r$ lift $\bar x_1,\dots, \bar x_r$ and $x_{r+1},\dots,x_{r+s}$ live in $IM$. This is indeed possible, since we can take any lifts $x_1,\dots,x_r$ and any family $(x_{r+1},\dots,x_{r+s})$ of generators of the finitely generated $B$-module $IM$.
\end{proof}

\begin{rem}
	As in Proposition~\ref{prop:h^fflat via Hilb}, one can also show that the map $\Quot_d(X/S,\sO_X^E) \to \COH_d(X/S)$ is a universal $\A^1$-equivalence on affine schemes for any infinite set $E$.
\end{rem}

There is a canonical map $\Gr_{d-1}(\A^n)\to \Hilb_{d}(\A^n)$ sending a surjection $\sO_T^n \to \sE$ to the surjection $\Sym_{\sO_T}(\sO_T^n)\to \sO_T\oplus\sE$, where $\sO_T\oplus\sE$ is the square-zero extension of $\sO_T$ by $\sE$.

\begin{cor}\label{cor:Gr-Hilb}
		The map $\Gr_{d-1}(\A^\infty) \to \Hilb_{d}(\A^\infty)$ is an $\A^1$-equivalence on affine schemes.
\end{cor}

\begin{proof}
	There is a commutative square
	\[
	\begin{tikzcd}
				\Gr_{d-1}(\A^\infty) \ar{r}\ar{d} & \Hilb_{d}(\A^\infty)\ar{d}\\
		\Vect_{d-1} \ar{r}{\alpha} & \FFLAT_{d}\rlap.
	\end{tikzcd}
	\]
	The vertical maps are $\A^1$-equivalences on affine schemes by Proposition \ref{prop:quot} and Proposition~\ref{prop:h^fflat via Hilb}, and the lower horizontal map is an $\A^1$-equivalence by Theorem~\ref{thm:stack-equivalence}.
\end{proof}

    In \cite[Section~1]{segal1977K-homology}, Segal constructs for a compact topological space $X$ its ``labeled'' configuration space $F(X_+)$, which is an $\Einfty$-space whose points are finite sets of points of $X$, labeled by jointly orthogonal real vector spaces embedded into $\bR^\infty$. The space $F(X_+)$ has a canonical forgetful map to the total symmetric power of $X$, which sends $\{V_x \}_{x \in S}$ to $\sum_{x \in S} \dim (V_x) \cdot [x]$, where $S$ is a finite set of points of $X$. 

     For a smooth quasi-projective variety $X$ over a field $k$, the ind-scheme $\Hilb(\A^\infty_X)$ can be thought of as a geometric analogue of Segal's labeled configuration space of $X$, with vector spaces replaced by finite algebras.
   To see this, recall the Hilbert--Chow
   morphism~\cite[Theorem~2.16]{Bertin__punctual_Hilbert_schemes}
    \[
        \rho_X\colon \Hilb_{d}(X) \to \Sym^d(X),
    \]
    where $\Sym^d(X)$ is the quotient of $X^{d}$ by the naturally acting
    symmetric group on $d$ letters.
	 On points, the morphism $\rho_X$ sends a finite closed subscheme $Z \subset X$ to the support of $Z$ counted with
    multiplicities.
    Fix another quasi-projective $k$-scheme $Y$ and consider the composition
    \[
        \rho_{X\times Y/X}:=\mathrm{pr}_1\circ\rho_{X \times Y}\colon \Hilb_d(X\times Y) \to \Sym^d(X \times Y)\to \Sym^d(X).
    \]
    For $x\in X$ a closed point, denote by $\Hilb_d(X \times Y, x)$ the fiber
    $\rho_{X\times Y/X}^{-1}(d[x])$. This is the scheme parametrizing subschemes of $X\times Y$ of degree $d$ supported on $\{x\}\times Y$.
    By construction \cite[Theorem~2.16, see~(2.34)]{Bertin__punctual_Hilbert_schemes}, for every
    quasi-projective $X$ and $Y$ the morphism $\rho_{X\times Y/X}$
    factors through $\hfflat_d(X)$.

    \begin{lem}\label{lem:fiberwiseChow}
        Let $\Gamma = \sum_{i=1}^r \mu_i [x_i]$ be a $k$-point of
        $\Sym^d(X)$, where $x_1, \ldots ,x_r\in X$ are nonsingular
        $k$-points.
        Then the fiber $\rho_{X\times Y/X}^{-1}(\Gamma)$ is isomorphic to
        $\prod_{i=1}^r \Hilb_{\mu_i}(\A^{\dim_{x_i} X}\times Y, 0)$.
    \end{lem}

    \begin{proof}
        First let us reduce to the case where $r=1$. Choose for every point $x_i$ an open neighborhood $U_i$ that does not contain the other points. Then for every $k$-scheme $T$ and a family $Z\subset X\times Y\times T$ corresponding to a $T$-point of the fiber we can write $Z$ as
        \[Z\simeq \coprod_i Z\times_X U_i\]
        where $Z\times_X U_i$ is supported on $\{x_i\}\times Y\times T$. Therefore
        \[\rho_{X\times Y/X}^{-1}(\Gamma) \simeq \prod_i \rho_{X\times Y/X}^{-1}(\mu_i [x_i]).\]

        For the rest of the proof we assume $\Gamma= d [x]$, so that the fiber is equal to $\Hilb_d(X\times Y,x)$. We need to prove
        \[\Hilb_d(X\times Y,x)\simeq \Hilb_d(\A^{\dim_xX}\times Y,0).\]
        To do so we will prove a stronger statement. Let $e = d\cdot \dim_x(X)$ and let $X_{x,e}=\Spec \sO_{X,x}/\mm_{X,x}^{e}$ be the $e$-thickening of $x\in X$. We will prove
        \[\Hilb_d(X\times Y,x)\simeq \Hilb_d(X_{x,e}\times Y,x).\]
        The claim then follows from the fact that $X_{x,e}\simeq \A^{\dim_xX}_{0,e}$ since $x$ is a nonsingular point of $X$. Without loss of generality we can replace $X$ by an open neighborhood of $x$ and we thus assume $X$ is affine.

        Summing up, we need to show that for every $k$-scheme $T$ and every $T$-family $Z\subset X\times Y\times T$ of degree $d$ supported on $\{x\}\times Y\times T$, the subscheme $Z$ in fact lies in $X_{x,e}\times Y\times T$.  It is enough to show that the $d$-th power of every $f\in \mm_{X,x}$ vanishes when pulled back to $Z$. Shrinking $X$ further if necessary, we may view such an $f$ as a map $f\colon X\to \A^1$ with $f(x)=0$. Consider the composition
        \[f' = \Sym^d(f) \circ \rho_{X\times Y/X} \colon \Hilb_d(X\times Y)\to \Sym^d(X)\to \Sym^d(\A^1) \simeq \A^d,\]
        where the last isomorphism is given by elementary symmetric polynomials and maps $d[0]$ to $0\in \A^d$.
        Since $f(x)=0$, the map $f'$ restricted to $\Hilb_d(X\times Y,x)$ is zero.
        
        Let $T\to \Hilb_d(X\times Y)$ be a map classifying a family $Z\subset X\times Y\times T$. By the description of the Hilbert--Chow morphism in terms of linearized determinants, see~\cite[Corollary~2.18]{Bertin__punctual_Hilbert_schemes} and~\cite[2.4, p.9]{Iversen__linear_determinants}, the restriction of $f'$ to $T$ computes exactly the non-leading coefficients of the characteristic polynomial of the $\sO_T$-linear endomorphism of $\sO_Z$ given by multiplication by $f\in\sO(X)$. Applying this to a family $Z$ in $\Hilb_d(X\times Y,x)$, we see that all these coefficients are $0$ and so $f^d=0$ on $Z$ by the Cayley--Hamilton theorem.
    \end{proof}

    Let $\Hilb_{d}(\A^n\times \A^{\infty}, 0)$ be the ind-scheme $\colim_{m}\Hilb_d(\A^n\times\A^{m}, 0)$. We deduce from Lemma~\ref{lem:fiberwiseChow} that for a smooth quasi-projective $n$-dimensional $k$-scheme $X$, the fiber of
    \begin{equation} \label{eq:hilbert-chow}
    \rho_{\A^\infty_X/X}\colon \Hilb_d(X\times \A^\infty) \to \Sym^d(X)
    \end{equation} over a $k$-point $\sum_{i=1}^r \mu_i [x_i]$ is isomorphic to $\prod_{i=1}^r \Hilb_{\mu_i}(\A^n\times\A^{\infty}, 0)$.

    The
    restrictions of the forgetful maps $\Hilb_d(\A^{n}\times\A^m)\to \FFLAT_d$ induce
    a forgetful map
    \[
        \Hilb_d(\A^n\times\A^{\infty}, 0)\to \FFLAT_d.
    \]
    \begin{lem}\label{lem:punctualHilb}
        The forgetful map $\Hilb_d(\A^n\times \A^{\infty}, 0)\to \FFLAT_d$ is an $\A^1$-equivalence on affine schemes.
    \end{lem}
	 
    \begin{proof}
        As remarked before Lemma~\ref{lem:fiberwiseChow}, the map $\rho_{\A^n\times \A^\infty/\A^n}$ factors through
        $\hfflat_d(\A^n)$.
        Denote by $\hfflat_d(\A^n, 0)\hook \hfflat_d(\A^n)$ the
        fiber over $d[0]\in \Sym^d(\A^n)$ of the induced map. By Proposition~\ref{prop:h^fflat via Hilb}(\ref{it:potencialInfinity})
        we know that \[\Hilb_d(\A^n\times\A^{\infty}, 0)\to
        \hfflat_d(\A^n,0)\] is an $\A^1$-equivalence on affine schemes. It remains to prove that the forgetful map
        $\alpha'\colon\hfflat_d(\A^n, 0)\to
        \FFLAT_d$ is an $\A^1$-equivalence.
        The map $\alpha'$ has a natural section
        $\beta'\colon\FFLAT_d\to \hfflat(\A^n, 0)$ that equips a family with
        a constant map to $0\in \A^n$.
        The monoid $(\A^1, \cdot)$ acts on $\A^n$
        and so induces a map
        \begin{equation}\label{eq:HilbChowhomotopy}
            \A^1 \times \hfflat_d(\A^n)\to \hfflat_d(\A^n)
        \end{equation}
        that sends $(t, U\leftarrow Z\xrightarrow{\varphi} \A^n)$ to $U\leftarrow
        Z\xrightarrow{t\varphi} \A^n$. Since $d[0]$ is a $\A^1$-fixed point, the
        map~\eqref{eq:HilbChowhomotopy} restricts to a map $\A^1 \times
        \hfflat_d(\A^n, 0)\to \hfflat_d(\A^n, 0)$ which gives an
        $\A^1$-homotopy $\beta'\alpha' \simeq \id_{\FFLAT_d}$. This shows that $\beta'$
        is an inverse of $\alpha'$ up to $\A^1$-homotopy.
    \end{proof}

    \begin{cor}\label{cor:be like segal}
     Let $X$ be a smooth quasi-projective $k$-scheme. Then the fiber of the map
     \[\rho_{\A^\infty_X/X}\colon\Hilb(\A^\infty_X)\to \Sym(X)\]
     over a $k$-point $\sum_i d_i[x_i] \in \Sym (X)$ is $\A^1$-equivalent on affine schemes to $\prod_i\FFLAT_{d_i}$, and hence also to $\prod_i\Vect_{d_i-1}$.
    \end{cor}
    \begin{proof}
     This follows from the combination of Lemma~\ref{lem:fiberwiseChow}, Lemma~\ref{lem:punctualHilb}, and Theorem~\ref{thm:stack-equivalence}.
    \end{proof}

\section{The effective motivic K-theory spectrum}
\label{sec:kgl}

Fix a base scheme $S$.
Let $\Span^\fflat(\Sch_S)$ be the $(2,1)$-category whose objects are $S$-schemes and whose morphisms are spans $X\leftarrow Z\to Y$ with $Z\to X$ finite locally free.
It is semiadditive and has a symmetric monoidal structure given by the cartesian product of $S$-schemes. We equip the $\infty$-category $\Pre_\Sigma(\Span^\fflat(\Sch_S))$ with the induced symmetric monoidal structure given by Day convolution.
Recall that commutative monoids for the Day convolution are precisely right-lax symmetric monoidal functors \cite[Proposition 2.12]{Saul}.
The unit of this symmetric monoidal structure is the presheaf $\FFLAT_S$ (the restriction of $\FFLAT$ to $\Sch_S$), which therefore has a unique structure of commutative monoid.

We claim that the presheaf $\Vect_S$ also has a structure of commutative monoid in $\Pre_\Sigma(\Span^\fflat(\Sch_S))$, with transfers given by the pushforward of finite locally free sheaves and multiplication given by the tensor product.
To see this, we apply the symmetric monoidal unfurling construction of \cite[Corollary 7.8.1]{BarwickGlasmanShah} to the category of finite locally free sheaves fibered over $\Sch_S^\op$ to obtain a right-lax symmetric monoidal functor
$\Vect_S\colon \Span^\fflat(\Sch_S)^\op \to \Spc$.\footnote{The unfurling construction in \emph{loc.\ cit.}\ requires a symmetric monoidal Waldhausen bicartesian fibration over $(\Sch_S^\op,\fflat)$ as input. The symmetric monoidal bicartesian fibration is here the fibered category of finite locally free sheaves with the external tensor product. One can equip it with the minimal Waldhausen structure, whose cofibrations are summand inclusions of locally free sheaves, but in fact the Waldhausen structure is not used to construct the unfurling.}
The forgetful map $\eta\colon \FFLAT_S\to \Vect_S$ can then be promoted to a morphism of commutative monoids in $\Pre_\Sigma(\Span^\fflat(\Sch_S))$, namely the unique such morphism.

Let $\Span^\fr(\Sch_S)$ be the symmetric monoidal $\infty$-category of framed correspondences constructed in \cite[\sectsign 4]{deloop1}. Its objects are $S$-schemes and its morphisms are spans
\[
\begin{tikzcd}
   & Z \ar[swap]{ld}{f}\ar{rd}{g} & \\
  X &   & Y\rlap,
\end{tikzcd}
\]
where $f$ is finite syntomic and equipped with a trivialization of its cotangent complex $\sL_f$ in the $K$-theory space $K(Z)$. There is a symmetric monoidal forgetful functor $\Span^\fr(\Sch_S)\to\Span^\fflat(\Sch_S)$ \cite[4.3.15]{deloop1}, allowing us to regard $\eta\colon\FFLAT_S\to \Vect_S$ as a morphism of commutative monoids in $\Pre_\Sigma(\Span^\fr(\Sch_S))$. Below we also denote by $\FFLAT_S$ and $\Vect_S$ the restriction of these presheaves to $\Span^\fr(\Sm_S)$, which is the full subcategory of $\Span^\fr(\Sm_S)$ spanned by the smooth $S$-schemes.

Recall that any presheaf with framed transfers on $\Sm_S$ gives rise to a motivic spectrum via the functor
\[
\Sigma^\infty_{\T,\fr}\colon \Pre_\Sigma(\Span^\fr(\Sm_S)) \to \SH^\fr(S)\simeq\SH(S),
\]
where the equivalence is \cite[Theorem 18]{framed-loc}.

Let $\beta\colon (\P^1,\infty) \to \Vect^\gp$ be the morphism corresponding to the formal difference of locally free sheaves $\sO_{\P^1}-\sO_{\P^1}(-1)$. For any scheme $S$, $\beta$ induces a canonical element in $\pi_{2,1}(\Sigma^\infty_{\T,\fr}\Vect_S)$.
Following \cite[Section 3]{HoyoisCdh}, we shall say that a $\Vect_S^\gp$-module $\sF$ in some presentable $\infty$-category with an action of $\Pre_\Sigma(\Span^\fr(\Sm_S))$ is \emph{$\beta$-periodic} if the morphism
\[
\sF\to \Omega_{\P^1} \sF
\]
induced by $\beta$ is an equivalence. For example, algebraic $K$-theory (as a pointed presheaf on $\Sm_S$) is $\beta$-periodic, by the projective bundle formula.

\begin{prop}\label{prop:KGL-Vect}
	For any scheme $S$, there is an equivalence of $\Einfty$-ring spectra
	\[
	\KGL_S\simeq (\Sigma^\infty_{\T,\fr}\Vect_S)[\beta^{-1}]
	\]
	in $\SH(S)\simeq\SH^\fr(S)$.
\end{prop}

\begin{proof}
	By \cite[Proposition 3.2]{HoyoisCdh}, the functors $\Omega^\infty_\T$ and $\Omega^\infty_{\T,\fr}$ restrict to equivalences of symmetric monoidal $\infty$-categories of $\beta$-periodic modules
	\begin{gather*}
	P_\beta\Mod_{\Vect^\gp}(\H(S)_*) \simeq P_\beta\Mod_{\Vect^\gp}(\SH(S)),
	\\
	P_\beta\Mod_{\Vect^\gp}(\H^\fr(S)) \simeq P_\beta\Mod_{\Vect^\gp}(\SH^\fr(S)).
	\end{gather*}
	Under the former equivalence, the $\Einfty$-ring spectrum $\KGL_S$ corresponds to the pointed motivic space $\Omega^\infty KH$ \cite[Proposition 2.14]{Cisinski}, which is the $\beta$-periodization of $L_\mot\Vect^\gp$ \cite[Proposition 4.10]{HoyoisCdh}.
	
	For $\sF$ a $\Vect^\gp$-module in either $\H(S)_*$ or $\H^\fr(S)$, its $\beta$-periodization $P_\beta\sF$ is computed as a sequential colimit of the presheaves $\Omega^n_\T\sF$ \cite[Theorem 3.8 and Lemma 4.9]{HoyoisCdh}. Since the forgetful functor $\H^\fr(S)\to\H(S)_*$ commutes with sequential colimits and with $\Omega^n_\T$, the square
	\[
	\begin{tikzcd}
		\Mod_{\Vect^\gp}(\H^\fr(S)) \ar{r}{P_\beta} \ar{d} & P_\beta\Mod_{\Vect^\gp}(\SH^\fr(S)) \ar{d} \\
		\Mod_{\Vect^\gp}(\H(S)_*) \ar{r}{P_\beta} & P_\beta\Mod_{\Vect^\gp}(\SH(S))
	\end{tikzcd}
	\]
	of right-lax symmetric monoidal functors commutes. The commutativity of this square on the framed motivic space $L_\mot\Vect^\gp$ gives the desired equivalence.
\end{proof}

Let $\mathrm{kgl}_S\in\SH(S)^\veff$ denote the very effective cover of $\KGL_S$ \cite[Definition 5.5]{SpitzweckOstvaer}. We shall only consider the spectrum $\mathrm{kgl}_S$ when $S$ is regular over a field (i.e., regular and equicharacteristic), in which case it coincides with the effective cover of $\KGL_S$ (this is equivalent to the vanishing of the negative K-theory of fields, by the characterization of very effectivity in terms of homotopy sheaves \cite[\sectsign 3]{BachmannSlices}).

\begin{cor}\label{cor:kgl-vect}
	If $S$ is regular over a field, there is an equivalence of $\Einfty$-ring spectra
	\[
	\mathrm{kgl}_S\simeq \Sigma^\infty_{\T,\fr}\Vect_S
	\]
	in $\SH(S)\simeq\SH^\fr(S)$.
\end{cor}

\begin{proof}
	By Theorem~\ref{prop:KGL-Vect}, there is a canonical $\Einfty$-map $\Sigma^\infty_{\T,\fr}\Vect_S\to \kgl_S$. The assertion that it is an equivalence is local on $S$, so we may assume $S$ affine.
	By Popescu's theorem \cite[Tag 07GC]{stacks}, $S$ is then a cofiltered limit of smooth $k$-schemes for some perfect field $k$. Since pro-smooth base change preserves very effective covers, we can assume that $S$ is the spectrum of a perfect field $k$. In this case, by the motivic recognition principle \cite[Theorem 3.5.14]{deloop1}, the very effective cover of a motivic spectrum $E$ is $\Sigma^\infty_{\T,\fr}\Omega^\infty_{\T,\fr}(E)$, and for $X$ a framed motivic space we have $\Omega^\infty_{\T,\fr}\Sigma^\infty_{\T,\fr}X \simeq X^\gp$. Since $L_\mot\Vect^\gp$ is already $\beta$-periodic on regular schemes (as it agrees with the $K$-theory presheaf $K$), we deduce that $\Sigma^\infty_{\T,\fr}\Vect_k$ is the very effective cover of $(\Sigma^\infty_{\T,\fr}\Vect_k)[\beta^{-1}]$, so the corollary follows from Proposition~\ref{prop:KGL-Vect}.
\end{proof}

\begin{rem}
	By \cite[Proposition B.3]{norms}, Corollary~\ref{cor:kgl-vect} holds for a finite-dimensional scheme $S$ if and only if, for every $s\in S$, the canonical map $s^*(\mathrm{kgl}_S)\to \mathrm{kgl}_{\kappa(s)}$ is an equivalence. One expects this to be true at least if $S$ is smooth over a Dedekind scheme.
\end{rem}

\begin{thm}\label{thm:kgl-fflat}
If $S$ is regular over a field, there is an equivalence of $\Einfty$-ring spectra
\[
\mathrm{kgl}_S \simeq \Sigma^\infty_{\T,\fr} \FFLAT_S
\]
in $\SH(S)\simeq \SH^\fr(S)$.
\end{thm}

\begin{proof}
Combine Corollary~\ref{cor:kgl-vect} and Theorem~\ref{thm:main-gp}.
\end{proof}

\begin{rem}
	One can lift the Bott element $\beta\colon (\P^1,\infty)\to\Vect^\gp$ to $\FFLAT^\gp$ (for example, by taking the formal difference of square-zero extensions $\sO\oplus\sO - \sO\oplus\sO(-1)$). Combining Proposition~\ref{prop:KGL-Vect} and Theorem~\ref{thm:main-gp}, we then have an equivalence of $\Einfty$-ring spectra \[\KGL_S\simeq (\Sigma^\infty_{\T,\fr}\FFLAT_S)[\beta^{-1}]\] for any scheme $S$.
\end{rem}

	Let $\H^\fflat(S)\subset \Pre_\Sigma(\Span^\fflat(\Sm_S))$ be the full subcategory spanned by the $\A^1$-invariant Nisnevich sheaves, and let $\SH^\fflat(S)$ the $\infty$-category of $\T$-spectra in $\H^\fflat(S)_*$.
	Theorem~\ref{thm:kgl-fflat} formally leads to an adjunction
	\[
	\Mod_{\kgl}(\SH(S)) \rightleftarrows \SH^\fflat(S),
	\]
	where the left adjoint is symmetric monoidal. In \cite{BachmannFFlatCancellation}, Bachmann proves that this adjunction induces an equivalence
	\[
	\Mod_{\kgl}(\SH(k))[\tfrac 1e] \simeq \SH^\fflat(k)[\tfrac 1e],
	\]
	for $k$ a perfect field of exponential characteristic $e$~\cite[Corollary~4.3]{BachmannFFlatCancellation}. 
	
	The results in~\cite{BachmannFFlatCancellation} (which in turn depend on Theorem~\ref{thm:kgl-fflat}) have the following consequence.
	
	\begin{prop} \label{prop: kgl otimes X}
	Let $k$ be a perfect field of exponential characteristic $e$ and $X$ a smooth $k$-scheme. Then there is an equivalence
    \[\Omega^\infty_{\T} (\kgl \otimes \Sigma^\infty_{\T} X_+) [\tfrac 1e] \simeq L_\Zar\Lhtp \hfflat(X)^\gp [\tfrac 1e]. \]
	 If moreover $\kgl \otimes \Sigma^\infty_\T X_+$ is a dualizable $\kgl$-module (for example if $X$ is proper), there is an equivalence
	 \[\Omega^\infty_{\T} (\kgl \otimes \Sigma^\infty_{\T} X_+) \simeq L_\Zar\Lhtp \hfflat(X)^\gp. \]
	\end{prop}

\begin{proof}
By the cancellation theorem for finite locally free correspondences \cite[Theorem~3.5 and Proposition~2.13]{BachmannFFlatCancellation}, for every $\sF\in\H^\fflat(k)$ there is a natural equivalence
\[\sF^\gp \simeq \Omega^\infty_{\T} \Sigma^\infty_{\T} \sF\]
in $\H^\fflat(k)$.
Moreover, if $\sF\in\Pre_\Sigma(\Span^\fflat(\Sm_k))^\gp$, then $L_\mot\sF \simeq L_\Zar\Lhtp\sF$ by \cite[Theorem 3.4.11]{deloop1}.
The first statement now follows from the equivalence of $\infty$-categories $\Mod_{\kgl}(\SH(k))[\tfrac 1e] \simeq\SH^\fflat(k)[\tfrac 1e]$ \cite[Corollary~4.3]{BachmannFFlatCancellation}. 
Without inverting $e$, the unit of the adjunction $\Mod_{\kgl}(\SH(k))\rightleftarrows\SH^\fflat(k)$ is an equivalence on the unit object \cite[Theorem 4.2]{BachmannFFlatCancellation}, hence on any dualizable object. This implies the second statement.
\end{proof}	

\begin{cor}\label{cor: kgl otimes X}
	Let $k$ be a perfect field and $X$ a smooth separated $k$-scheme. Suppose that $\kgl \otimes \Sigma^\infty_\T X_+$ is a dualizable $\kgl$-module (for example $X$ is proper). Then there is an equivalence
	 \[\Omega^\infty_{\T} (\kgl \otimes \Sigma^\infty_{\T} X_+) \simeq (L_\Zar\Lhtp \Hilb(\A^\infty_X))^\gp. \]
\end{cor}

\begin{proof}
	Combine Propositions~\ref{prop: kgl otimes X} and \ref{prop:h^fflat via Hilb}.
\end{proof}

    In light of Corollary~\ref{cor:be like segal}, we can think of Corollary~\ref{cor: kgl otimes X} as a geometric analogue of Segal's result~\cite[Proposition~1.1]{segal1977K-homology}, expressing $\mathrm{ko}$-homology of a topological space $X$ as the homotopy groups of the labeled configuration space $F(X)$, see Section~\ref{sec: A^infty times X}. Here $\mathrm{ko}$ is the connective cover of the real topological K-theory spectrum $\mathrm{KO}$.
    
\begin{rem} 
		 Let $k$ be a perfect field and $\HZ \in \SH(k)$ the motivic Eilenberg--Mac Lane spectrum over $k$. 
		 Since $s_0(\mathbf{1})\simeq s_0(\kgl)\simeq \HZ$ \cite{Levine:2008}, there is a \emph{unique} morphism of $\Einfty$-ring spectra $\kgl \to \HZ$.
 Let $\Cor(k)$ be Voevodsky's category of finite correspondences over $k$ \cite[Definition~1.5]{mvw}.
 By \cite[\sectsign 5.3]{deloop1}, there is a symmetric monoidal functor
 \begin{equation*}\label{eq:fflat->Cor}
 \Span^\fflat(\Sm_k^\mathrm{ft,sep})\to \Cor(k)
 \end{equation*}
 sending a span $X \xleftarrow{f} Z \xrightarrow{g} Y$ to the cycle $(f, g)_*[Z]$ on $X\times Y$. 
  It induces an adjunction
 \[
 \SH^\fflat(k)\rightleftarrows \DM(k)
 \]
 where the left adjoint is symmetric monoidal. Recall that $\HZ$ is by definition the image of the unit object of $\DM(k)$ by the forgetful functor $\DM(k)\to \SH(k)$. The unit of the above adjunction therefore induces the (unique) morphism of $\Einfty$-ring spectra $\kgl\to \HZ$, and we obtain a commutative square of adjunctions
 \[
 \begin{tikzcd}
	 \Mod_{\kgl}(\SH(k)) \ar[shift left=.5ex]{r} \ar[shift right=.5ex]{d} & \Mod_{\HZ}(\SH(k)) \ar[shift left=.5ex]{l} \ar[shift right=.5ex]{d} \\ 
 	\SH^\fflat(k) \ar[shift left=.5ex]{r} \ar[shift right=.5ex]{u} & \DM(k)\rlap. \ar[shift left=.5ex]{l} \ar[shift right=.5ex]{u}
 \end{tikzcd}
 \]
 The vertical adjunctions are known to be equivalences after inverting the exponential characteristic $e$ (see~\cite{Rondigs:2008} and~\cite[Theorem~5.8]{HoyoisKellyOstvaerSteenrodAtP}).
 For $X\in\Sm_k^\mathrm{ft,sep}$, the canonical map
 \begin{equation}\label{eq:kgl->HZ}
 \Omega^\infty_{\T} (\kgl \otimes \Sigma^\infty_{\T} X_+)\to \Omega^\infty_{\T} (\HZ \otimes \Sigma^\infty_{\T} X_+)
 \end{equation}
 can thus be identified, after inverting $e$ (or if $X$ is proper), with the motivic localization of the map
 \[
 \hfflat(X)^\gp \to h^\Cor(X),
 \]
 where $h^\Cor(X)$ is the presheaf represented by $X$ on $\Cor(k)$. 
 Furthermore, if $X$ is quasi-projective, there is a commutative diagram
 \[
 \begin{tikzcd}
 	\Hilb(\A^\infty_X) \ar{r}{\rho_{\A^\infty_X/X}} \ar{d} & \Sym(X)^\gp \\
	\hfflat(X) \ar{r} & h^\Cor(X) \ar{u}
 \end{tikzcd}
 \]
 where $\rho_{\A^\infty_X/X}$ is the Hilbert--Chow morphism \eqref{eq:hilbert-chow}, $\Hilb(\A^\infty_X)\to \hfflat(X)$ is a motivic equivalence (Proposition~\ref{prop:h^fflat via Hilb}), and $h^\Cor(X)\to\Sym(X)^\gp$ is an isomorphism after inverting $e$ \cite[Theorem 6.8]{Suslin:1996}. Altogether, the map~\eqref{eq:kgl->HZ} corresponds, up to inverting $e$, to the group completion of the Hilbert--Chow morphism $\rho_{\A^\infty_X/X}\colon \Hilb(\A^\infty_X) \to \Sym(X)$.
\end{rem}

\section{Comparison with algebraic cobordism}
\label{sec:MGL->KGL}

Let $\FSYN_S$ be the presheaf with framed transfers on $\Sm_S$ associating to $X$ the groupoid of finite syntomic $X$-schemes.
The forgetful map $\FSYN_S \to \Vect_S$ induces a map of $\Einfty$-ring spectra \[\Sigma^\infty_{\T,\fr} \FSYN_S \to \Sigma^\infty_{\T,\fr} \Vect_S.\] By~\cite[Theorem~3.4.1]{deloop3}, $\Sigma^\infty_{\T,\fr} \FSYN_S \simeq \MGL_S$ is the algebraic cobordism spectrum of Voevodsky. By Proposition~\ref{prop:KGL-Vect}, $(\Sigma^\infty_{\T,\fr} \Vect_S)[\beta^{-1}]$ is identified with  $\KGL_S$, so we get an $\Einfty$-map $\MGL_S \to \KGL_S$. The goal of this section is to prove that this map is an $\Einfty$-refinement of the usual orientation of $\KGL_S$. 

Let $\PMGL_S$ be the spectrum $\bigoplus_{n\in\Z}\Sigma^{2n,n}\MGL_S$ with the $\Einfty$-ring structure from \cite[Theorem 16.19]{norms}. The commutative monoid structure on $\PMGL_S$ in the homotopy category $\h\SH(S)$ is determined by that of $\MGL_S$, and it classifies \emph{periodic oriented ring spectra}, i.e., oriented ring spectra $E$ with a given unit in $\pi_{2,1}(E)$.

By \cite[Theorem 3.4.1]{deloop3}, there is an equivalence of ($\Z$-graded) $\Einfty$-ring spectra
\begin{equation}\label{eqn:PMGL}
(\Sigma^\infty_{\T,\fr}\FQSM_S)[u^{-1}] \simeq \PMGL_S,
\end{equation}
where $u\in \pi_{2,1}(\Sigma^\infty_{\T,\fr}\FQSM_S^1) \simeq \pi_{2,1}(\Sigma^{2,1}\MGL_S)$ is the suspension of the unit of $\MGL_S$. 
Here, $\FQSM_S(X)$ is the $\infty$-groupoid of finite quasi-smooth derived $X$-schemes, with the $\Einfty$-semiring structure given by the sum and the product of $X$-schemes. Recall that a morphism of derived schemes $f\colon Y\to X$ is \emph{quasi-smooth} if it locally of finite presentation with perfect cotangent complex $\sL_f$ of Tor-amplitude $\leq 1$, and \emph{finite} if the underlying morphism of classical schemes is finite. Recall also that $f\colon Y\to X$ is quasi-smooth if and only if, locally on $Y$, it is the vanishing locus of finitely many functions on a smooth $X$-scheme \cite[2.3.13]{KhanVCD}; in particular, quasi-smooth morphisms are locally of finite Tor-amplitude.

Under the equivalence~\eqref{eqn:PMGL}, Thom classes of vector bundles have the following geometric description.
If $\sE$ is a locally free sheaf of rank $n$ over a smooth $S$-scheme $X$, its Thom class in $\MGL^{2n,n}_X(\bV(\sE))$ is represented by the zero section $X\hook \bV(\sE)$ in $\FQSM_S^n(\bV(\sE))$. Indeed, this holds by construction of the equivalence $\Sigma^\infty_{\T,\fr}\FQSM_S^n\simeq \Sigma^{2n,n}\MGL_S$, cf.\ \cite[Construction 3.1.1]{deloop3}.

\begin{prop}\label{prop:KGL-orientation}
	Let $S$ be a scheme.
The $\Einfty$-map $\MGL_S\to\KGL_S$ induced by $\FSYN_S \to \Vect_S$ is the standard orientation of $\KGL_S$.
\end{prop}

\begin{proof}
	By \cite[Theorem 6.1.3.2]{SAG}, the structure sheaf of any finite quasi-smooth $S$-scheme is a perfect $\mathcal O_S$-module, since quasi-smooth morphisms are locally of finite presentation and locally of finite Tor-amplitude.
		There is therefore a commutative square of commutative monoids in presheaves with framed transfers
		\[
		\begin{tikzcd}
		\FSYN_S \ar{r} \ar{d} & \FQSM_S \ar{d} \\
		\Vect_S \ar{r} & \Perf_S,
		\end{tikzcd}
		\]
		where $\Perf_S(X)$ is the $\infty$-groupoid of perfect complexes on $X$ and the right vertical morphism $\FQSM_S\to \Perf_S$ forgets the algebra structure. 
		By Proposition~\ref{prop:KGL-Vect}, there is an $\Einfty$-map of presheaves with framed transfers $\Vect_S\to \Omega^\infty_{\T,\fr}\KGL_S$, which factors through $L_\Nis\Vect_S^\gp$ (since Nisnevich sheafification is compatible with framed transfers, see \cite[Proposition 3.2.4]{deloop1}). 
		On the other hand, the algebraic $K$-theory presheaf $K$ has finite locally free transfers induced by the pushforward functors between the $\infty$-categories of perfect complexes, which extend the pushforwards of vector bundles. Since $\Vect^\gp\to K$ is a Nisnevich-local equivalence, we deduce that the $\Einfty$-map $\Vect_S\to L_\Nis\Vect_S^\gp$ factors through $\Perf_S$ in $\Pre_\Sigma(\Span^\fr(\Sm_S))$. Hence, the composite $\FSYN_S\to\Vect_S\to \Omega^\infty_{\T,\fr}\KGL_S$ factors through $\FQSM_S$. By adjunction and \cite[Theorem 3.4.1]{deloop3}, we obtain a sequence of $\Einfty$-maps
	 \[\MGL_S\to \bigoplus_{n\geq 0}\Sigma^{2n,n}\MGL_S\to \KGL_S.\]
	 Recall that the standard periodic orientation of $\KGL_S$ is such that the Thom class of a finite locally free sheaf $\sE$ over $X$ is represented by the perfect $\sO_{\bV(\sE)}$-module $\sO_X$, which is precisely the image of the zero section $X\hook \bV(\sE)$ by the forgetful map $\FQSM_S(\bV(\sE))\to \Perf_S(\bV(\sE))$. Hence, the map $\Sigma^{2n,n}\MGL_S\to\KGL_S$ constructed above preserves Thom classes of locally free sheaves of rank $n$. In particular, $\Sigma^{2,1}\MGL_S\to\KGL_S$ sends $u$ to the Bott element $\beta$, which is the Thom class of $\sO_S$, and we obtain an induced $\Einfty$-map
	 \[
	 \PMGL_S\to\KGL_S.
	 \]
	This map preserves Thom classes by construction, i.e., it induces the standard periodic orientation of $\KGL_S$.
\end{proof}

\begin{rem}
	For $S=\Spec \bC$, the $\Einfty$-map $\PMGL_S\to \KGL_S$ constructed above realizes to the folklore $\Einfty$-map $MUP \to KU$ alluded to in \cite[page 3]{HahnYuan}. In particular, it differs from the Gepner--Snaith $\Einfty$-map $\PMGL_S\to\KGL_S$ constructed in \cite{GepnerSnaith}, which uses a different $\Einfty$-ring structure on $\PMGL_S$, see~\cite[Remark 16.20]{norms}.
\end{rem}

\begin{cor}\label{cor:KGL-orientation}
	Suppose that $S$ is regular over a field. Under the equivalence $\kgl_S\simeq\Sigma^\infty_{\T,\fr}\FFLAT_S$ of Theorem~\ref{thm:kgl-fflat}, the standard orientation $\MGL_S\to \mathrm{kgl}_S$ is $\Sigma^\infty_{\T,\fr}$ of the forgetful map $\FSYN_S\to \FFLAT_S$.
\end{cor}

\section{The Hilbert scheme of finite lci schemes of degree 3}\label{sec:deg 3}

The open subset of lci subschemes in the Hilbert scheme
has a rich homotopy type. In particular, in positive characteristic,
the limiting space $\Hilb_\infty^\lci(\A^\infty)$ has
the $\A^1$-homotopy type of the unit component of
the infinite loop space associated
to algebraic cobordism,
$\Omega^\infty_\T \MGL$ \cite[Theorem 1.2]{deloop4}. In characteristic
zero, these two spaces at least have isomorphic motives
in the derived category of motives,
$\DM(k)$. Hopkins conjectured that the motive of
$\Omega^\infty_\T \MGL$ is pure Tate, but that remains
open. By definition, a pure Tate motive is a direct sum
of the motives $\Z\{a\}:=\Z(a)[2a]$ for integers $a$.

A possible approach to Hopkins' conjecture is through the space
$\Hilb_\infty^\lci(\A^\infty)$. In that direction, we have
the following result.

\begin{thm}\label{thm:degree3}
The motive of $\Hilb_3^{\lci}(\A^\infty)$ over $\Z$ is pure Tate in $\DM(\Z)$.
The Chow ring
of $\Hilb_3^{\lci}(\A^\infty)$ is
\[\Z[c_1,c_2]/(c_2(-2c_1^2 + 9c_2)),\]
where $|c_i|=i$.
\end{thm}

\begin{rem}
As observed in~\cite[Remark~1.8]{deloop4}, the motive of $\Hilb_3^{\lci}(\A^n)$ is \emph{not} pure Tate for any finite $n \geq 2$. 
\end{rem}

Here, $\DM(\Z)$ denotes Spitzweck's $\infty$-category of motives over $\Spec \Z$ \cite[Chapter 9]{SpitzweckHZ}.
If $S$ is a regular scheme of dimension $\leq 1$, the $\infty$-category $\DM(S)$ categorifies Bloch--Levine motivic cohomology, in the sense that the presheaf $X\mapsto \Maps_{\DM(S)}(M(X),\Z\{n\})$ on smooth $S$-schemes is the Zariski sheafification of Bloch's cycle complex $X\mapsto z^n(X,*)$. By \cite[Theorem 1.2(5)]{Geisser}, we have
\begin{equation*}
\pi_0\Maps_{\DM(S)}(\Z\{a\},\Z\{b\}) = \begin{cases}
\Z & \text{if $a=b$,} \\
\Pic(S) & \text{if $a+1=b$,} \\
0 & \text{otherwise.}
\end{cases}
\end{equation*}
In particular, the homotopy category of pure Tate motives over $\Spec\Z$ is equivalent to the category of free graded abelian groups. For $X\in\Pre(\Sm_S)$, we will write
\[
CH^a(X) = H^{2a}(X,\Z(a)) = \pi_0\Maps_{\DM(S)}(M(X), \Z\{a\}).
\] 
These are the classical Chow groups when $S$ is semilocal and $X$ is a smooth $S$-scheme of finite type (since Bloch's cycle complex is already a Zariski sheaf in this case \cite[Theorem 1.7]{LevineLocalization}).
By \cite[Proposition 10.1]{SpitzweckHZ} and \cite[Proposition 6.2]{NSO2}, we have a canonical isomorphism
\[
CH^*(\BGL_n) \simeq \Z[c_1,\dotsc,c_n]
\]
with $c_i$ in degree $i$.

\begin{proof}[Proof of Theorem \ref{thm:degree3}]
Let $X_n=\Hilb_3(\A^n)$ over $\Z$, and let $Y_n\subset X_n$ be the closed complement of $\Hilb_3^\lci(\A^n)$. Let $X=\colim_n X_n$ and $Y=\colim_n Y_n$; we want
to describe the motive of $X-Y=\colim_n (X_n-Y_n)$.
For an algebraically closed field $k$, every non-lci finite $k$-scheme of degree $3$ is isomorphic to $\Spec k[x,y]/(x,y)^2$ \cite[Table 1]{Poonen__isotypes}, and any two embeddings into $\A^n_k$ differ by an affine transformation.
It follows that the obvious morphism
$f\colon W_n:=\Gr_2(\A^n)\times \A^n\to X_n$
is injective on $k$-points with image $Y_n(k)$
for every algebraically closed field $k$. In particular,
$f$ is quasi-finite; by definition, it suffices to check this over fields,
and then it suffices to prove this over algebraically closed fields
by descent \cite[Tag 02VI]{stacks}. 

Also, $f$ is proper, as it is the restriction
over $\Hilb_3(\A^n)$
of a $\Gr_2(\A^n)$-bundle over $\P^n$ mapping to $\Hilb_3(\P^n)$.
Since $f$ is proper and quasi-finite, it is finite \cite[Tag 02OG]{stacks}.
It is also unramified, meaning that the sheaf of relative differentials
$\Omega^1_{W_n/X_n}$ is zero. Indeed, the sheaf is coherent, and so
this can be checked over fields by Nakayama's lemma. After reducing further
to an algebraically closed base field,
one can check that $f$ is injective on tangent spaces
using the
calculation below of the tangent space to the Hilbert scheme $X_n$;
so $f$ is unramified.
Since $f$ is finite, unramified, and injective on points
over algebraically closed fields,
$f$ is a closed immersion \cite[Tag 04DG]{stacks}.
That is, $Y_n$ with reduced scheme structure
is isomorphic to $\Gr_2(\A^n)\times \A^n$.

By Corollary~\ref{cor:Gr-Hilb},
the inclusion $Y\to X$ is a motivic equivalence. 
Both $Y_n$ and $X_n$ are smooth over $\Z$, with $Y_n$
of codimension 4 in $X_n$.
So we have a cofiber sequence
\[M(X-Y)\to M(X)\to M(Y)\{4\}\]
in $\DM(\Z)$, where $M(Y)\{4\}$ and $M(X)$ are pure Tate. So, to show that
$M(X-Y)$ is pure Tate as well, it suffices to show
that the morphism $M(X) \to M(Y)\{4\}$ is the projection
onto a summand. A morphism between pure Tate motives
is determined by what it does on Chow groups $CH^*$,
so it suffices to show that the pushforward homomorphism
\[CH^{*-4}(Y)\to CH^*(X)\]
is a split injection of abelian groups.
Since the restriction $CH^*(X)\to CH^*(Y)$ is an isomorphism,
it suffices to show that the composed map
\[CH^{*-4}(Y) \to CH^*(X) \to CH^*(Y)\]
is split injective.
This composition is multiplication by the top Chern class
of the normal bundle, $c_4(N_{Y/X})$
\cite[Corollary 6.3]{Fulton}. So it suffices to show that
this class is nonzero modulo every prime number $p$
in $CH^*(Y) = CH^*(\BGL_2) = \Z[c_1,c_2]$. (Indeed, then multiplication
by $c_4(N_{Y/X})$ is injective modulo $p$ for every prime number $p$,
and (looking at each graded piece)
a homomorphism of finitely generated free abelian groups
which is injective modulo every prime is split injective.) We will show
that $c_4(N_{Y/X})$ is $c_2(-2c_1^2 + 9c_2),$
which will complete the proof.

Let $U_n$ be the open subset of $X_n$ of degree $3$ subschemes
not contained in any affine line. Then $Y_n$ is contained in $U_n$,
and so the normal bundle of $Y_n$ in $X_n$ can also be viewed
as its normal bundle in $U_n$.
The point is that each point of $U_n$, viewed as a degree $3$ subscheme
of $\A^n$, spans an affine plane, and so $U_n$ is fibered over the scheme
of affine planes in $\A^n$ with fiber $U_2$. 
More precisely, let $\Aff_n$ be the group scheme of affine automorphisms of $\A^n$, and let $A(m,n)$ be the scheme of affine embeddings $\A^m\hook \A^n$. Then $\Aff_n$ acts on $X_n$, the subschemes $U_n$ and $Y_n$ are $\Aff_n$-invariant, and the closed immersion $Y_n\hook U_n$ can be identified with
\[
Y_2\times_{\Aff_2} A(2,n) \hook U_2\times_{\Aff_2}A(2,n).
\]
Hence, the normal bundle of $Y_n \hook X_n$ is extended from the $\Aff_2$-equivariant normal bundle of $Y_2\hook U_2$. In other words, it is the pullback of $N_{Y_2/X_2}$ by the projection
\[
Y_n \simeq Y_2\times_{\Aff_2} A(2,n) \to [Y_2/\Aff_2].
\]
Since $Y_2\simeq\A^2$, there is an equivalence of stacks $[Y_2/\Aff_2] \simeq \BGL_2$.
Moreover, the induced map $Y=\colim_n Y_n\to \BGL_2$ is a motivic equivalence, inducing the canonical motivic equivalence $\Gr_2 \to \BGL_2$ (by Proposition~\ref{prop:quot}). So our goal is to compute the top Chern class of the $\GL_2$-representation given by the (rank 4) normal bundle to $Y_2$ in $X_2=\Hilb_3(\A^2)$
at the $\Z$-point $S = \Spec\Z[x,y]/(x^2,xy,y^2) \subset\A^2$ in $Y_2$.

Classically~\cite[3.2.1]{Sernesi__Deformations} the tangent bundle to $\Hilb_3(\A^2)$ at $[S]$ is
\[H^0(S, N_{S/\A^2}) = \Hom_{\sO_S}(I/I^2, \sO_S),\]
where $I = (x^2, xy, y^2) \subset \Z[x,y]$
and $\sO_S = \Z[x,y]/I$. By direct inspection, this $\Z$-module is
free of rank $6$ and isomorphic to
\[\Hom_\Z(\Z\{x^2,xy,y^2\}, \Z\{x,y\})
= \Hom(S^2(V^*),V^*),\]
where $V=\A^2$, so $V^*$ is the space
of linear functions on $\A^2$, with basis $x,y$.

To compute the Chern classes of this representation of $\GL_2$, we can work
over $\bC$, so $\GL_2$ becomes linearly reductive.
The Clebsch--Gordan formula gives that
the tangent space of $X_2=\Hilb_3(\A^2)$
at $[S]$ is $V \oplus (S^3(V) \otimes \det(V)^*)$
as a $\GL_2$-representation \cite[Proposition II.5.5]{BtD}.
In particular its only $2$-dimensional subrepresentation is $V$ which thus
coincides with the tangent space to $Y_2$. So the normal space
to $Y_2$ in $X_2$ at $[S]$ is $S^3(V) \otimes \det(V)^*$
as a representation of $\GL_2$.
The top Chern class $c_4$ of $S^3(V) \otimes\det(V)^*$
is $c_2(-2c_1^2 + 9c_2)$ in $CH^*(\BGL_2)= \Z[c_1,c_2]$,
as one checks by the splitting principle
(the fact that the Chow ring of $\BGL_2$ injects into that
of $B\G_m^2$).
That is what we want.
\end{proof}

\section{Stability theorems for the Hilbert scheme}
\label{sec:stability}

We have shown that $\Hilb_d(\A^\infty)$
has the $\A^1$-homotopy type of $\BGL_{d-1}$.
We now prove corresponding stability theorems for the inclusion $\Hilb_d(\A^n)\subset \Hilb_d(\A^\infty)$.
There are several notions of connectivity available in $\A^1$-homotopy theory, and we shall prove that the inclusion $\Hilb_d(\A^n)\subset \Hilb_d(\A^\infty)$ is highly connected in various senses (Theorem~\ref{thm:Hilb-stability}). In particular, we show that it induces an isomorphism in motivic cohomology in weights $\leq n-d+1$ (Corollary~\ref{cor:motivic-stability}).

\begin{rem}\label{rem:connectivity-numbering}
We use the convention that, for $n\geq -1$, a morphism $f$ in $\Spc$ is $n$-connected if its fibers are $n$-connected (meaning that $\pi_i=0$ for $i\leq n$). It follows, in particular, that $f$ is surjective on $\pi_0$. In the older literature, such a map would be called $(n+1)$-connected. Our numbering is becoming more common, as in \cite[Definition 6.55]{Morel} or \cite[Remark 1.6]{RezkBlakersMassey}.
\end{rem}

We will say that a morphism $f\colon Y\to X$ in $\Pre(\Sm_S)$ is \emph{$\A^1$-$n$-connected} if $L_\mot(f)$ is $n$-connected as a morphism of Nisnevich sheaves (i.e., $n$-connected on Nisnevich stalks). We will frequently use Morel's $\A^1$-connectivity theorem \cite[Theorem 6.38]{Morel}, which states that if $k$ is a perfect field and $X\in \Shv_\Nis(\Sm_k)$ is $n$-connected, then $X$ is $\A^1$-$n$-connected.
Recall also that a motivic spectrum is \emph{very $n$-effective} if it can be obtained using colimits and extensions from the image of the functor $\Sigma^n_\T\Sigma^\infty_\T$.

\begin{thm}\label{thm:Hilb-stability}
	Let $n\geq d\geq 0$.
	\begin{enumerate}
		\item The morphism \[\Hilb_d(\A^n)(\bC)\to \Hilb_d(\A^\infty)(\bC)\] is $(2n-2d+2)$-connected and the morphism \[\Hilb_d(\A^n)(\bR)\to \Hilb_d(\A^\infty)(\bR)\] is $(n-d)$-connected.
		\item If $k$ is a field,
the morphism \[\Sigma\Hilb_d(\A^n_k)\to \Sigma\Hilb_d(\A^\infty_k)\] is $\A^1$-$(n-d+1)$-connected. Here, $\Sigma X$ denotes the pushout $*\sqcup_X*$ in $\Pre(\Sm_k)$.
		\item If $k$ is a field, the cofiber of 
		\[\Sigma^\infty_\T\Hilb_d(\A^n_k)_+\to \Sigma^\infty_\T\Hilb_d(\A^\infty_k)_+\]
		in $\SH(k)$ is very $(n-d+2)$-effective. 
	\end{enumerate}
\end{thm}

The proof of Theorem~\ref{thm:Hilb-stability} is given at the end of this section.

\begin{cor}\label{cor:motivic-stability}
	Let $k$ be a field and $A$ an abelian group.
	\begin{enumerate}
 		\item The homomorphism of motivic cohomology groups
	\[
	H^*(\BGL_{d-1},A(a))\to H^*(\Hilb_d(\A^n),A(a)) 
	\]
	is an isomorphism for all $a\leq n-d+1$.
 	\item Let $i\geq 0$. The homomorphism of Milnor--Witt motivic cohomology groups
	\[
	 H^{2a+i}(\BGL_{d-1},\tilde A(a))\to H^{2a+i}(\Hilb_d(\A^n),\tilde A(a)) 
	\]
	is an isomorphism if $a+i\leq n-d$ and injective if $a+i=n-d+1$.
	\end{enumerate} 
\end{cor}

\begin{proof}
	Both statements follow from Theorem~\ref{thm:Hilb-stability}(3) and the following facts:
	if $E\in\SH(k)$ is very $r$-effective, then $H^*(E,A(a))=0$ for $a<r$ and $H^{2a+i}(E,\tilde A(a))=0$ for $i\geq 0$ and $a+i<r$.
	The former is because the $1$-effective cover of the motivic Eilenberg--MacLane spectrum $HA$ is zero, since motivic cohomology vanishes in negative weights. The latter is because the very $1$-effective cover of $H\tilde A$ is zero, since in negative weights Milnor--Witt motivic cohomology is given by the cohomology of the Witt sheaf, which vanishes in negative degrees \cite[Example~6.1.5]{BachmannCalmes}.
\end{proof}

\begin{rem}
	In Theorem~\ref{thm:Hilb-stability}(2,3) and in Corollary~\ref{cor:motivic-stability}, we regard $\Hilb_d(\A^n_k)$ as a presheaf on smooth $k$-schemes. Since $\Hilb_d(\A^n_k)$ is in fact a scheme, one can also define its motivic cohomology and Milnor--Witt motivic cohomology more intrinsically by working in $\SH(\Hilb_d(\A^n_k))$. By standard arguments using resolution of singularities or alterations, these two ways of defining the cohomology of non-smooth schemes are known to agree with $\Z[\tfrac 1e]$-linear coefficients, where $e$ is the exponential characteristic of $k$.
\end{rem}

Let $\sZ_d\to\FFLAT_d$ be the universal finite locally free scheme of degree $d$ and let $\sV_{d}(\A^n)$ be the vector bundle over $\FFLAT_d$ defined by 
\[
\sV_{d}(\A^n) = \Hom_{\FFLAT_d}(\sZ_d,\A^n_{\FFLAT_d}).
\]
Then $\Hilb_d(\A^n)\subset \sV_d(\A^n)$ is the open substack of closed immersions $\sZ_d\to \A^n_{\FFLAT_d}$. The key geometric input to all our stability theorems is the following observation:

\begin{lem}\label{lem:codim}
	The closed complement of $\Hilb_d(\A^n)$ in $\sV_d(\A^n)$ has codimension at least $n-d+2$ in every fiber over $\FFLAT_d$.
\end{lem}

\begin{proof}
	Let $k$ be a field and $S=\Spec R$ a finite $k$-scheme of degree $d$.
	We must show that the closed complement of $\Emb_k(S,\A^n_k)$ in $\Hom_k(S,\A^n_k)\simeq \A^{nd}_k$ has codimension at least $n-d+2$.
	A $k$-morphism $S\to \A^n_k$
	can be viewed as a $k$-algebra homomorphism
	$k[x_1,\ldots,x_n]\to R$, or as a $k$-linear map
	$k\{x_1,\ldots,x_n\}\to R$. 
	Such a morphism $S\to \A^n_k$ is an embedding if and only if the corresponding homomorphism $k[x_1,\ldots,x_n]\to R$ is surjective.
	In particular, $\Emb_k(S,\A^n_k)$ contains the open subset of $k$-linear maps $k\{x_1,\ldots,x_n\}\to R$ such that $k\{x_1,\ldots,x_n\}\to R/(k\cdot 1)$ is surjective. The subset of non-surjective linear maps $k^n\to k^r$
has codimension $n-r+1$; indeed, the space of all linear maps
has dimension $nr$, whereas any non-surjective map lands in some
hyperplane in $k^r$ (corresponding to a point in $\P^{r-1}$),
and so the subset of non-surjective maps
has dimension $n(r-1)+(r-1)=nr-(n-r+1)$.
Thus the complement of $\Emb_k(S,\A^n_k)$ in $\Hom_k(S,\A^n_k)$
has codimension at least $n-(d-1)+1=n-d+2$, as we want.
\end{proof}

\begin{rem}\label{rem:optimal}
The bound in Lemma \ref{lem:codim} is optimal for $n\geq d-2$, the worst
case being the square-zero extension algebra
$$R=k[x_1,\ldots,x_{d-1}]/(x_ix_j:\text{all }i,j)$$
over a field $k$.
In that case, a $k$-algebra homomorphism
$k[x_1,\ldots,x_n]\to R$ is surjective if and only if the
$k$-linear map $k\{ x_1,\ldots,x_n\}\to R/(k\cdot 1)$ is surjective.
As a result, the subset of non-surjective $k$-algebra homomorphisms
$k[x_1,\ldots,x_n]\to R$ has codimension equal to $n-(d-1)+1=n-d+2$.
\end{rem}

\begin{prop}\label{prop:topological-connectivity}
	Let $k$ be a subfield of $\C$, $X$ an algebraic stack of finite type over $k$, $V\to X$ a vector bundle, and $r\geq 0$. Let $U\subset V$ be an open substack whose complement has codimension at least $r$ in every fiber of $V\to X$.
	Then $U(\bC)\to X(\bC)$ is $(2r-2)$-connected in the classical topology.
If moreover $k\subset\bR$ and $X$ is an algebraic space, then $U(\bR)\to X(\bR)$ is $(r-2)$-connected in the classical topology.
\end{prop}

\begin{proof}
	Let $X_\bullet\to X$ be a smooth hypercover where each $X_i$ is an affine scheme. Then $X(\bC)$ is the colimit of $X_\bullet(\bC)$ in $\Spc$. If $X$ is an algebraic space, we can choose $X_\bullet\to X$ to be a Nisnevich hypercover by \cite[Proposition 5.7.6]{GrusonRaynaud}, in which case also $X(\bR)$ is the colimit of $X_\bullet(\bR)$ in $\Spc$.
	Since simplicial colimits preserve connectivity, we can assume that $X$ is an affine scheme.

We give the rest of the proof for $\bC$, but exactly the same argument applies to $\bR$.
	Let $0\leq j<2r$; we need to show that $\pi_j U(\bC)\to \pi_jX(\bC)$
	is surjective (for any base points $u_0$ in $U(\bC)$ and $x_0$ in $X(\bC))$,
	and that $\pi_{j-1}U(\bC)\to \pi_{j-1}X(\bC)$ is injective.
	Let $h$ be a continuous map $S^j\to X(\bC)$; we will show that
$h$ has arbitrarily small perturbations that lift to $U(\bC)$. Although the
perturbations we construct need not keep the base point fixed, they will
send the base point $s_0$ of $S^j$ near $u_0$ in $U(\bC)$;
this implies the surjectivity of $\pi_j (U(\bC),u_0)\to \pi_j(X(\bC),x_0)$,
since $U(\bC)$ is locally contractible.

By the triangulation of real semialgebraic sets,
there are semialgebraic triangulations of $S^j$ and $X(\bC)$
\cite[section 1]{hironaka1975triangulations}. By simplicial approximation,
after subdividing the triangulation of $S^j$,
we can perturb $h$ to a simplicial map (still called $h$), which in particular
is semialgebraic. By \cite[Theorem 6]{FernandoGhiloni}, we can perturb $h$
to a $C^1$ semialgebraic map. (This makes sense even though $X(\bC)$
is singular, using local embeddings of $X(\bC)$ into affine space.
Any two embeddings are equivalent by complex analytic maps,
and so the notion of ``$C^1$'' is well-defined.)
The pullback of a $C^1$-vector bundle on $X(\C)$ along $h$
is a $C^1$-vector bundle on $S^j$, for example by considering transition
functions: a composite of $C^1$ maps is $C^1$.
As a result, the pullback
	of the map $U(\bC)\to X(\bC)$ to $S^j$ is an open subset of a $C^1$
semialgebraic real vector bundle $p\colon W\to S^j$,
	with complement a semialgebraic closed subset $Z\subset W$ that has real codimension
	at least $2r$ in every fiber. It follows that
$Z$ has real codimension
	at least $2r$ in $W$. Here $W$ is a $C^1$ manifold,
and $Z$ is a finite union of (possibly noncompact) $C^1$ submanifolds
of real codimension at least $2r$. (The smooth locus of $Z$ is a
(possibly noncompact) $C^1$ manifold,
and its complement has lower dimension; so we can
decompose $Z$ by induction on dimension.)
Choose a $C^1$ section $e$ of the vector
bundle $W\to S^j$ such that $e(s_0)=u_0$. Since $j<2r$,
the transversality theorem for $C^1$ manifolds \cite[Theorem 3.2.1]{Hirsch}
implies that $e$ 
can be perturbed to a $C^1$-section disjoint from $Z$. That is,
the map $h\colon S^j\to X(\bC)$ lifts to $U(\bC)$, as we want.
The same argument (transversality
	applied to the pullback vector bundle) also proves, for $j<2r$,
	that every continuous map $S^{j-1}\to U(\bC)$
	such that the composed map to $X(\bC)$ extends to $D^j$
has arbitrarily small perturbations that
	extend to a continuous map $D^j\to U(\bC)$.
	So $U(\bC)\to X(\bC)$ is $(2r-2)$-connected, as we want.
\end{proof}

\begin{lem}\label{lem:A1connectivity}
	Let $k$ be a perfect field, $f\colon Y\to X$ a morphism in $\Pre(\Sm_k)$, and $n\geq -1$. 
	If $f$ is $\A^1$-$n$-connected, then $\cofib(f)$ is $\A^1$-$(n+1)$-connected.
	The converse holds if $X$ and $Y$ are $\A^1$-$1$-connected. 
\end{lem}

\begin{proof}
	Let $C$ be the cofiber of $L_\mot(f)$ in $\Shv_\Nis(\Sm_k)$.
The class of $n$-connected maps is closed under cobase change \cite[Corollary 6.5.1.17]{HTT}. Therefore,
if $L_\mot(f)$ is $n$-connected, then $*\to C$ is $n$-connected, which is equivalent to $C$ being $(n+1)$-connected \cite[Proposition 6.5.1.20]{HTT}.
By the $\A^1$-connectivity theorem, $C$ is $\A^1$-$(n+1)$-connected.
Conversely, suppose that $X$ and $Y$ are $\A^1$-$1$-connected and that $C$ is $\A^1$-$(n+1)$-connected. Let $F$ be the fiber of $L_\mot(f)$. We will prove that $C$ is in fact $(n+1)$-connected, which implies (by considering the Nisnevich stalks) that $F$ is $n$-connected \cite[Theorem IV.7.13]{Whitehead}. Note that $C$ is $1$-connected since $L_\mot(X)$ and $L_\mot(Y)$ are. By induction, we may assume that $C$ is $n$-connected and $F$ is $(n-1)$-connected. By Blakers--Massey, the canonical map $F\to \Omega C$ is then $n$-connected. In particular, $\pi_{n+1}(C)$ is isomorphic to $\pi_n(F)$, hence is strictly $\A^1$-invariant. By \cite[Corollary 6.60]{Morel}, we have $\pi_{n+1}(C) \simeq \pi_{n+1}L_\mot(C) = 0$, as desired.
\end{proof}

\begin{lem}\label{lem:A1connectivity2}
	Let $k$ be a perfect field, $X$ a smooth $k$-scheme, and $Z\subset X$ a closed subscheme of codimension $\geq r$. Then $\Sigma(X/(X-Z))$ is $\A^1$-$r$-connected.
\end{lem}

\begin{proof}
	We can assume $X$ quasi-compact.
	If $Z$ is smooth, then $X/(X-Z)$ is $\A^1$-$(r-1)$-connected by the purity isomorphism \cite[\sectsign 3, Theorem 2.23]{MV}.
	 In general, the result is trivial if $r=0$. If $r\geq 1$, then $X/(X-Z)$ is $\A^1$-connected by \cite[Lemma 6.1.4]{morel-connectivity}. We can therefore assume $r\geq 2$.
	 Since $k$ is perfect, there exists a filtration
	 \[
	 \emptyset = Z_0\subset Z_1\subset \dotsb\subset Z_n = Z
	 \]
	 by closed subschemes such that $Z_i-Z_{i-1}$ is smooth.
	 We prove the result by induction on the length $n$ of the filtration, the case $n=0$ being trivial.
	 We therefore assume that $\Sigma(X/(X-Z_{n-1}))$ is $\A^1$-$r$-connected. 
	 Since $\Sigma(X/(X-Z))$ and $\Sigma((X-Z_{n-1})/(X-Z))$ are $\A^1$-$1$-connected, Lemma~\ref{lem:A1connectivity} and the cofiber sequence
	\begin{equation*}
	\frac{X-Z_{n-1}}{X-Z} \to \frac {X}{X-Z} \to \frac {X}{X-Z_{n-1}}
	\end{equation*}
	imply that the morphism
	 $\Sigma((X-Z_{n-1})/(X-Z)) \to \Sigma(X/(X-Z))$ is $\A^1$-$(r-1)$-connected. Since the source is $\A^1$-$r$-connected (by the smooth case), it follows that $\Sigma(X/(X-Z))$ is $\A^1$-$r$-connected.
\end{proof}

\begin{prop}\label{prop:general-connectivity}
	Let $k$ be a perfect field, $X\colon \Sm_k^\op\to \Spc$ a presheaf, $V\to X$ a vector bundle, and $r\geq 0$. Let $U\subset V$ be an open subpresheaf such that, for every finite field extension $k'/k$ and every $\alpha\in X(k')$, the closed complement of $\alpha^*(U)$ in $\alpha^*(V)$ has codimension at least $r$.
	\begin{enumerate}
		\item The morphism $\Sigma^2U\to \Sigma^2X$ is $\A^1$-$r$-connected. If moreover $U$ and $X$ are $\A^1$-connected, then the morphism $\Sigma U\to \Sigma X$ is $\A^1$-$(r-1)$-connected.
		\item The cofiber of $\Sigma^\infty_\T U_+\to \Sigma^\infty_\T X_+$ in $\SH(k)$ is very $r$-effective.
	\end{enumerate}
\end{prop}

\begin{proof}
	(1) By Lemma~\ref{lem:A1connectivity}, it suffices to show that $\Sigma\cofib(U\to X)$ is $\A^1$-$r$-connected. 
	Colimits of pointed objects preserve connectivity, hence $\A^1$-connectivity by the $\A^1$-connectivity theorem, so we are reduced by universality of colimits to the case $X\in\Sm_k$. Since $V\to X$ is an $\A^1$-equivalence and $V-U$ has codimension $\geq r$ in $V$, the result follows from Lemma~\ref{lem:A1connectivity2}.
	
	(2) 
	As in the proof of (1), we can assume $X\in\Sm_k$. Let $Z$ be the closed complement of $U$ in $V$. If $Z$ is smooth, the claim follows from the purity isomorphism. In general, since $k$ is perfect, one can stratify $Z$ by smooth subschemes of codimension $\geq r$ in $V$ (cf.\ the proof of Lemma~\ref{lem:A1connectivity2}), thereby reducing to the case $Z$ smooth.
\end{proof}

\begin{lem}\label{lem:Hilb-connectivity}
	Let $k$ be a field and $n\geq d-1\geq 0$. Then $\Hilb_d(\A^n_{k})$ is $\A^1$-connected.
\end{lem}

In fact, $\Hilb_d(\A^n_k)$ is $\A^1$-connected for all $n$ and $d$, at least when $k$ is infinite,
as Totaro shows by a more elaborate argument \cite[Theorem 6.1]{TotaroMorse}.
Since $\A^1$-connectedness implies connectedness, this result recovers Hartshorne's theorem that $\Hilb_d(\P^n)$ is connected \cite{HartshorneHilb} (although Hartshorne's theorem also applies to nonconstant Hilbert polynomials).

\begin{proof} 
	By \cite[Lemma 6.1.3]{morel-connectivity} and \cite[Section 2, Corollary 3.22]{MV}, it suffices to show that $\Hilb_d(\A^n)(k)$ is nonempty and that any two points of $\Hilb_d(\A^n)(F)$ can be connected by a chain of affine lines over $F$, for every separable finitely generated field extension $F/k$. It is clear that $\Hilb_d(\A^n)(k)$ is nonempty.

    Fix coordinates $\A^n_F=\Spec F[x_1, \ldots ,x_n]$.
    Consider an $F$-point $[A]\in \Hilb_d(\A^n)(F)$ corresponding to a surjection $\pi\colon F[x_1, \ldots ,x_n]\to A$.
    We claim that $\pi$ can be connected by a chain of affine lines to a surjection such that the images of $1, x_1, \ldots ,x_{d-1}$ are linearly independent. We prove this by adjusting $\pi$ as necessary. If $1,x_1, \ldots ,x_{i-1}$ are linearly independent while $1,x_1, \ldots ,x_{i}$ are dependent, then
\[\pi|_{\hat{i}}\colon F[x_1, \ldots ,x_{i-1}, x_{i+1}, \ldots ,x_n]\to A\]
is a surjection. For $a\in A$ consider the $F[t]$-algebra homomorphism $\rho\colon F[x_{1}, \ldots ,x_n,t]\to A[t]$ defined by $\rho(x_j) = \pi(x_j)$ for $j\neq i$ and $\rho(x_i) = ta+(1-t)\pi(x_i)$. 
For every $\lambda\in \overline{F}$, the map $\rho_\lambda$ is surjective since it extends $\pi|_{\hat{i}}$. Hence, $\rho$ itself is surjective and it defines an $\A^1$-homotopy from $\pi = \rho_0$ to $\rho_1$ in $\Hilb_d(\A^n)$. If we choose $a\in A- F\{\pi(1),\pi(x_1), \ldots ,\pi(x_{i-1})\}$, then the images of $1, x_1, \ldots ,x_{i}$ by $\rho_1$ are linearly independent. Continuing by induction we prove the claim. Having $\pi$ such that $1, x_1, \ldots ,x_{d-1}$ span $A$, we use another $\A^1$-homotopy to assure $\pi(x_j) = 0$ for $j\geq d$.

    As in the proof of Theorem~\ref{thm:stack-equivalence}, consider the algebra $\Rees(A) := F\oplus tA[t]\subset A[t]$ which is finite locally free over $F[t]$. Define an $F[t]$-algebra homomorphism $\tilde{\pi}\colon
F[x_1,\ldots,x_n,t]\to\Rees(A)$ by $\tilde{\pi}(x_i) = t\pi(x_i)$
for every $i$.
Since $1, x_1, \ldots ,x_{d-1}$ span $A$ as an $F$-vector space,
the image of $\tilde{\pi}$ contains $t$ as well as $tA$, and so
$\tilde{\pi}$ is surjective.
Geometrically, $\tilde{\pi}$ corresponds to a morphism $\tilde{\pi}\colon \A^1_F\to \Hilb_d(\A^n)$ that links $[A] = \tilde{\pi}(1)$ with $\tilde{\pi}(0)$. The latter $F$-point corresponds to an embedding of Spec of the square-zero extension $F\oplus F^{d-1}$ at the origin of the $(d-1)$-dimensional subspace $x_{d}=\ldots=x_{n}=0$ in $\A^n$. In particular, it is independent of $[A]$, which proves that every $F$-point can be connected to a fixed one by a chain of affine lines over $F$, as we want.
\end{proof}

\begin{proof}[Proof of Theorem~\ref{thm:Hilb-stability}]
	In (2) and (3), we may replace the field $k$ by a perfect subfield, since the conclusions are preserved by essentially smooth base change.
	Recall that the forgetful map $\Hilb_d(\A^\infty)\to\FFLAT_d$ is a motivic equivalence, see Corollary~\ref{cor:Hilb=FFlat}.
	For the connectivity on complex points and for (2) and (3), we apply Propositions \ref{prop:topological-connectivity} and~\ref{prop:general-connectivity} with $X=\FFLAT_d$, $V=\sV_d(\A^n)$, $U=\Hilb_d(\A^n)$, and $r=n-d+2$, see Lemma~\ref{lem:codim}.
	For (2), note that both $\Hilb_d(\A^n)$ and $\FFLAT_d$ are $\A^1$-connected by Lemma~\ref{lem:Hilb-connectivity}. To prove the connectivity on real points, we pull back the above situation to the scheme $X_m=\Hilb_d(\A^m)$. By Proposition~\ref{prop:topological-connectivity}, $(U\times_XX_m)(\bR)\to X_m(\bR)$ is $(n-d)$-connected. Taking the colimit over $m$ proves the result, using the fact that $\Hilb_d(\A^\infty)\to\FFLAT_d$ is a universal $\A^1$-equivalence on affine schemes, see Corollary~\ref{cor:Hilb=FFlat}.
\end{proof}

\begin{rem}
	Note that the assumptions of Propositions \ref{prop:topological-connectivity} and~\ref{prop:general-connectivity} are preserved by any base change $X' \to X$. For example, by considering the open substack $\FSYN_d\subset \FFLAT_d$, we see that Theorem~\ref{thm:Hilb-stability} holds with $\Hilb_d$ replaced by $\Hilb_d^\mathrm{lci}$.
\end{rem}

\bibliographystyle{alphamod}
\let\mathbb=\mathbf
{\small
\bibliography{references}
}
\parskip 0pt
\end{document}

%% file: preamble.tex
\usepackage[utf8]{inputenc}
\usepackage{url}

\usepackage[expansion=false]{microtype}
\usepackage{amssymb,amsmath,amsopn,amsxtra,amsthm,amsfonts}
\usepackage{xr}
\usepackage[mathcal]{euscript}
\usepackage{mathalfa}
\usepackage{txfonts}
\usepackage{todonotes}

\usepackage[pdfusetitle,unicode]{hyperref}

\numberwithin{equation}{section}
\makeatletter
\let\c@equation\c@subsection

\makeatother

\newtheorem{cor}[subsection]{Corollary}
\newtheorem{lem}[subsection]{Lemma}
\newtheorem{prop}[subsection]{Proposition}

\newtheorem{thm}[subsection]{Theorem}

\theoremstyle{definition}

\newtheorem{rem}[subsection]{Remark}

\theoremstyle{remark}

\renewcommand{\eqref}[1]{(\ref{#1})}

\usepackage{tikz}
\usetikzlibrary{matrix}
\usepackage{tikz-cd}
\tikzset{shorten <>/.style={shorten >=#1,shorten <=#1}}

\usepackage{fixltx2e}

\usepackage{xspace}

\usepackage{xifthen}

\usepackage{xparse}

\newcommand{\nc}{\newcommand}
\nc{\renc}{\renewcommand}
\nc{\ssec}{\subsection}
\nc{\sssec}{\subsubsection}
\nc{\on}{\operatorname}
\nc{\term}[1]{#1\xspace}

\DeclareMathSymbol{A}{\mathalpha}{operators}{`A}
\DeclareMathSymbol{B}{\mathalpha}{operators}{`B}
\DeclareMathSymbol{C}{\mathalpha}{operators}{`C}
\DeclareMathSymbol{D}{\mathalpha}{operators}{`D}
\DeclareMathSymbol{E}{\mathalpha}{operators}{`E}
\DeclareMathSymbol{F}{\mathalpha}{operators}{`F}
\DeclareMathSymbol{G}{\mathalpha}{operators}{`G}
\DeclareMathSymbol{H}{\mathalpha}{operators}{`H}
\DeclareMathSymbol{I}{\mathalpha}{operators}{`I}
\DeclareMathSymbol{J}{\mathalpha}{operators}{`J}
\DeclareMathSymbol{K}{\mathalpha}{operators}{`K}
\DeclareMathSymbol{L}{\mathalpha}{operators}{`L}
\DeclareMathSymbol{M}{\mathalpha}{operators}{`M}
\DeclareMathSymbol{N}{\mathalpha}{operators}{`N}
\DeclareMathSymbol{O}{\mathalpha}{operators}{`O}
\DeclareMathSymbol{P}{\mathalpha}{operators}{`P}
\DeclareMathSymbol{Q}{\mathalpha}{operators}{`Q}
\DeclareMathSymbol{R}{\mathalpha}{operators}{`R}
\DeclareMathSymbol{S}{\mathalpha}{operators}{`S}
\DeclareMathSymbol{T}{\mathalpha}{operators}{`T}
\DeclareMathSymbol{U}{\mathalpha}{operators}{`U}
\DeclareMathSymbol{V}{\mathalpha}{operators}{`V}
\DeclareMathSymbol{W}{\mathalpha}{operators}{`W}
\DeclareMathSymbol{X}{\mathalpha}{operators}{`X}
\DeclareMathSymbol{Y}{\mathalpha}{operators}{`Y}
\DeclareMathSymbol{Z}{\mathalpha}{operators}{`Z}

\nc{\sA}{\ensuremath{\mathcal{A}}\xspace}
\nc{\sB}{\ensuremath{\mathcal{B}}\xspace}
\nc{\sC}{\ensuremath{\mathcal{C}}\xspace}
\nc{\sD}{\ensuremath{\mathcal{D}}\xspace}
\nc{\sE}{\ensuremath{\mathcal{E}}\xspace}
\nc{\sF}{\ensuremath{\mathcal{F}}\xspace}
\nc{\sG}{\ensuremath{\mathcal{G}}\xspace}
\nc{\sH}{\ensuremath{\mathcal{H}}\xspace}
\nc{\sI}{\ensuremath{\mathcal{I}}\xspace}
\nc{\sJ}{\ensuremath{\mathcal{J}}\xspace}
\nc{\sK}{\ensuremath{\mathcal{K}}\xspace}
\nc{\sL}{\ensuremath{\mathcal{L}}\xspace}
\nc{\sM}{\ensuremath{\mathcal{M}}\xspace}
\nc{\sN}{\ensuremath{\mathcal{N}}\xspace}
\nc{\sO}{\ensuremath{\mathcal{O}}\xspace}
\nc{\sP}{\ensuremath{\mathcal{P}}\xspace}
\nc{\sQ}{\ensuremath{\mathcal{Q}}\xspace}
\nc{\sR}{\ensuremath{\mathcal{R}}\xspace}
\nc{\sS}{\ensuremath{\mathcal{S}}\xspace}
\nc{\sT}{\ensuremath{\mathcal{T}}\xspace}
\nc{\sU}{\ensuremath{\mathcal{U}}\xspace}
\nc{\sV}{\ensuremath{\mathcal{V}}\xspace}
\nc{\sW}{\ensuremath{\mathcal{W}}\xspace}
\nc{\sX}{\ensuremath{\mathcal{X}}\xspace}
\nc{\sY}{\ensuremath{\mathcal{Y}}\xspace}
\nc{\sZ}{\ensuremath{\mathcal{Z}}\xspace}

\nc{\bA}{\ensuremath{\mathbf{A}}\xspace}
\nc{\bB}{\ensuremath{\mathbf{B}}\xspace}
\nc{\bC}{\ensuremath{\mathbf{C}}\xspace}
\nc{\bD}{\ensuremath{\mathbf{D}}\xspace}
\nc{\bE}{\ensuremath{\mathbf{E}}\xspace}
\nc{\bF}{\ensuremath{\mathbf{F}}\xspace}
\nc{\bG}{\ensuremath{\mathbf{G}}\xspace}
\nc{\bH}{\ensuremath{\mathbf{H}}\xspace}
\nc{\bI}{\ensuremath{\mathbf{I}}\xspace}
\nc{\bJ}{\ensuremath{\mathbf{J}}\xspace}
\nc{\bK}{\ensuremath{\mathbf{K}}\xspace}
\nc{\bL}{\ensuremath{\mathbf{L}}\xspace}
\nc{\bM}{\ensuremath{\mathbf{M}}\xspace}
\nc{\bN}{\ensuremath{\mathbf{N}}\xspace}
\nc{\bO}{\ensuremath{\mathbf{O}}\xspace}
\nc{\bP}{\ensuremath{\mathbf{P}}\xspace}
\nc{\bQ}{\ensuremath{\mathbf{Q}}\xspace}
\nc{\bR}{\ensuremath{\mathbf{R}}\xspace}
\nc{\bS}{\ensuremath{\mathbf{S}}\xspace}
\nc{\bT}{\ensuremath{\mathbf{T}}\xspace}
\nc{\bU}{\ensuremath{\mathbf{U}}\xspace}
\nc{\bV}{\ensuremath{\mathbf{V}}\xspace}
\nc{\bW}{\ensuremath{\mathbf{W}}\xspace}
\nc{\bX}{\ensuremath{\mathbf{X}}\xspace}
\nc{\bY}{\ensuremath{\mathbf{Y}}\xspace}
\nc{\bZ}{\ensuremath{\mathbf{Z}}\xspace}

\nc{\dA}{\ensuremath{\mathds{A}}\xspace}
\nc{\dB}{\ensuremath{\mathds{B}}\xspace}
\nc{\dC}{\ensuremath{\mathds{C}}\xspace}
\nc{\dD}{\ensuremath{\mathds{D}}\xspace}
\nc{\dE}{\ensuremath{\mathds{E}}\xspace}
\nc{\dF}{\ensuremath{\mathds{F}}\xspace}
\nc{\dG}{\ensuremath{\mathds{G}}\xspace}
\nc{\dH}{\ensuremath{\mathds{H}}\xspace}
\nc{\dI}{\ensuremath{\mathds{I}}\xspace}
\nc{\dJ}{\ensuremath{\mathds{J}}\xspace}
\nc{\dK}{\ensuremath{\mathds{K}}\xspace}
\nc{\dL}{\ensuremath{\mathds{L}}\xspace}
\nc{\dM}{\ensuremath{\mathds{M}}\xspace}
\nc{\dN}{\ensuremath{\mathds{N}}\xspace}
\nc{\dO}{\ensuremath{\mathds{O}}\xspace}
\nc{\dP}{\ensuremath{\mathds{P}}\xspace}
\nc{\dQ}{\ensuremath{\mathds{Q}}\xspace}
\nc{\dR}{\ensuremath{\mathds{R}}\xspace}
\nc{\dS}{\ensuremath{\mathds{S}}\xspace}
\nc{\dT}{\ensuremath{\mathds{T}}\xspace}
\nc{\dU}{\ensuremath{\mathds{U}}\xspace}
\nc{\dV}{\ensuremath{\mathds{V}}\xspace}
\nc{\dW}{\ensuremath{\mathds{W}}\xspace}
\nc{\dX}{\ensuremath{\mathds{X}}\xspace}
\nc{\dY}{\ensuremath{\mathds{Y}}\xspace}
\nc{\dZ}{\ensuremath{\mathds{Z}}\xspace}

\nc{\bbA}{\ensuremath{\mathbb{A}}\xspace}
\nc{\bbB}{\ensuremath{\mathbb{B}}\xspace}
\nc{\bbC}{\ensuremath{\mathbb{C}}\xspace}
\nc{\bbD}{\ensuremath{\mathbb{D}}\xspace}
\nc{\bbE}{\ensuremath{\mathbb{E}}\xspace}
\nc{\bbF}{\ensuremath{\mathbb{F}}\xspace}
\nc{\bbG}{\ensuremath{\mathbb{G}}\xspace}
\nc{\bbH}{\ensuremath{\mathbb{H}}\xspace}
\nc{\bbI}{\ensuremath{\mathbb{I}}\xspace}
\nc{\bbJ}{\ensuremath{\mathbb{J}}\xspace}
\nc{\bbK}{\ensuremath{\mathbb{K}}\xspace}
\nc{\bbL}{\ensuremath{\mathbb{L}}\xspace}
\nc{\bbM}{\ensuremath{\mathbb{M}}\xspace}
\nc{\bbN}{\ensuremath{\mathbb{N}}\xspace}
\nc{\bbO}{\ensuremath{\mathbb{O}}\xspace}
\nc{\bbP}{\ensuremath{\mathbb{P}}\xspace}
\nc{\bbQ}{\ensuremath{\mathbb{Q}}\xspace}
\nc{\bbR}{\ensuremath{\mathbb{R}}\xspace}
\nc{\bbS}{\ensuremath{\mathbb{S}}\xspace}
\nc{\bbT}{\ensuremath{\mathbb{T}}\xspace}
\nc{\bbU}{\ensuremath{\mathbb{U}}\xspace}
\nc{\bbV}{\ensuremath{\mathbb{V}}\xspace}
\nc{\bbW}{\ensuremath{\mathbb{W}}\xspace}
\nc{\bbX}{\ensuremath{\mathbb{X}}\xspace}
\nc{\bbY}{\ensuremath{\mathbb{Y}}\xspace}
\nc{\bbZ}{\ensuremath{\mathbb{Z}}\xspace}

\nc{\mrm}[1]{\ensuremath{\mathrm{#1}}\xspace}
\nc{\mbf}[1]{\ensuremath{\mathbf{#1}}\xspace}
\nc{\mcal}[1]{\ensuremath{\mathcal{#1}}\xspace}
\nc{\msc}[1]{\ensuremath{\mathscr{#1}}\xspace}

\renc{\bar}[1]{\overline{#1}}

\let\sectsign\S
\let\S\relax

\nc{\sub}{\subset}
\nc{\too}{\longrightarrow}
\nc{\hook}{\hookrightarrow}
\nc*{\hooklongrightarrow}{\ensuremath{\lhook\joinrel\relbar\joinrel\rightarrow}}
\nc{\hooklong}{\hooklongrightarrow}
\nc{\twoheadlongrightarrow}{\relbar\joinrel\twoheadrightarrow}
\nc{\shiso}{\approx}
\nc{\isoto}{\xrightarrow{\sim}}
\nc{\isofrom}{\xleftarrow{\sim}}
\renc{\ge}{\geqslant}
\renc{\le}{\leqslant}
\renc{\geq}{\geqslant}
\renc{\leq}{\leqslant}

\nc{\id}{\mathrm{id}}
\nc{\can}{\mathrm{can}}

\DeclareMathOperator{\Hom}{\mathrm{Hom}}
\nc{\uHom}{\underline{\smash{\Hom}}}
\DeclareMathOperator{\Maps}{\mathrm{Maps}}

\DeclareMathOperator{\End}{\mathrm{End}}
\DeclareMathOperator{\Sym}{\mathrm{Sym}}
\nc{\Pre}{\mathrm{PSh}{}}
\nc{\uEnd}{\underline{\smash{\End}}}

\renc{\lim}{\operatorname*{lim}}
\nc{\colim}{\operatorname*{colim}}
\nc{\Cofib}{\on{Cofib}}
\nc{\Fib}{\on{Fib}}
\nc{\initial}{\varnothing}
\nc{\op}{\mathrm{op}}


%% file: macros.tex
\nc{\bDelta}{\mbf{\Delta}}
\nc{\DM}{\mbf{DM}}
\nc{\eff}{\mathrm{eff}}
\nc{\veff}{\mathrm{veff}}
\nc{\cyc}{{\mrm{cyc}}}
\nc{\corr}{{\on{corr}}}
\nc{\fet}{{\mrm{f\acute et}}}
\nc{\fsyn}{{\mrm{fsyn}}}
\nc{\fflat}{{\mrm{fflat}}}
\nc{\syn}{{\mrm{syn}}}
\nc{\Perf}{\mathcal{P}\mrm{erf}}
\nc{\Pic}{\mrm{Pic}}
\nc{\Rees}{\mrm{Rees}}
\nc{\Cor}{\mrm{Cor}}
\nc{\perf}{\mrm{perf}}
\nc{\oblv}{\on{oblv}}
\nc{\exact}{\on{exact}}
\def\C{\mathbf{C}}
\nc{\F}{{\on{F}}}
\nc{\clopen}{{\mrm{clopen}}}
\nc{\B}{\mrm{B}}
\nc{\D}{\mrm{D}}
\nc{\Fin}{\on{Fin}}
\nc{\Cut}{\on{Cut}}
\nc{\Cart}{\on{Cart}}
\nc{\pairs}{\mathsf{pairs}}
\nc{\Pairs}{\mathrm{Pair}}
\nc{\Trip}{\mathrm{Trip}}
\nc{\Lab}{\mathrm{Lab}}
\nc{\coCart}{\mathrm{coCart}}
\nc{\RKE}{\mathrm{RKE}}
\nc{\strict}{\mathrm{strict}}
\nc{\Emb}{\mathrm{Emb}}
\nc{\EMB}{\mathcal{E}\mathrm{mb}}
\nc{\Split}{\mathrm{Split}}
\nc{\Set}{\mathrm{Set}}
\nc{\sSets}{\mathrm{sSets}}
\nc{\pb}{\mathrm{pb}}
\nc{\lci}{\mathrm{lci}}

\DeclareMathOperator*{\cofib}{cofib}
\nc{\diff}{\mrm{diff}}
\nc{\gp}{\mrm{gp}}
\nc{\chr}{\mrm{char}}
\nc{\mgp}{\mrm{mot-gp}}
\nc{\FSyn}{\mrm{FSyn}}
\nc{\FFlat}{{\mrm{FFlat}}}
\nc{\hfflat}{h^{\mrm{fflat}}}
\nc{\FEt}{\mrm{FEt}}
\nc{\Spc}{\mrm{Spc}}
\nc{\Ob}{\mrm{Ob}}
\nc{\Spt}{\mrm{Spt}}
\nc{\HZ}{\mathrm{H}\mathbf{Z}}
\nc{\T}{\bT}
\nc{\suspinf}{\Sigma^\infty}
\nc{\h}{\mrm{h}}
\nc{\uhom}{\underline{\mathrm{Hom}}}
\nc{\umap}{\underline{\mathrm{Maps}}}
\renc{\H}{\bH}
\nc{\Einfty}{{\sE_\infty}}
\nc{\Eone}{{\sE_1}}
\nc{\Stab}{\mrm{Stab}}
\nc{\lax}{{\mrm{lax}}}
\nc{\cocart}{{\mrm{cocart}}}
\nc{\Sch}{\mrm{Sch}}
\nc{\dSch}{\mrm{dSch}}
\nc{\Aff}{\mrm{Aff}}
\nc{\SmAff}{\mrm{SmAff}}
\nc{\dAff}{\mrm{dAff}}
\nc{\Fr}{\on{Fr}}
\nc{\A}{\mathbf{A}}
\nc{\N}{\mathbf{N}}
\nc{\Z}{\mathbf{Z}}
\nc{\Q}{\mathbf{Q}}
\nc{\Oo}{\mathcal{O}} 
\nc{\red}{{\on{red}}}
\nc{\Voev}{{\on{Voev}}}
\nc{\Corr}{\mrm{Corr}}
\nc{\Span}{\mathbf{Corr}}
\nc{\Gap}{\mrm{Gap}}
\nc{\Filt}{\mrm{Filt}}
\nc{\Corrfr}{\Corr^{\fr}}
\nc{\Corrvfr}{\Corr^{\Vfr}}
\nc{\Spec}{\on{Spec}}
\nc{\Sm}{\mrm{Sm}}
\nc{\QSm}{\mrm{QSm}}
\nc{\Gm}{\mathbf{G}_{\mrm{m}}}
\renc{\P}{\bP}
\nc{\Zar}{\mathrm{Zar}}
\nc{\et}{\mathrm{\acute et}}
\nc{\all}{\mathrm{all}}
\nc{\fold}{\mathrm{fold}}
\nc{\Fun}{\mathrm{Fun}}
\nc{\Ho}{\mathrm{Ho}}
\nc{\Segal}{\mathrm{Segal}}
\nc{\Mon}{\mrm{Mon}{}}
\nc{\Ab}{\mrm{Ab}}
\nc{\Gr}{\mrm{Gr}}
\nc{\Sh}{\on{Sh}}
\nc{\Shv}{\on{Shv}}
\nc{\M}{\mrm{M}}
\nc{\Lhtp}{L_{\A^1}}
\nc{\Lmot}{L_{\mrm{mot}}}
\nc{\mot}{\mrm{mot}}
\nc{\SH}{\mbf{SH}}
\nc{\RR}{\mbf{R}}
\nc{\CC}{\mbf{C}}
\nc{\Mod}{\mrm{Mod}}
\nc{\QCoh}{\mrm{QCoh}}
\nc{\MonUnit}{\mbf{1}}
\nc{\tr}{\on{tr}}
\nc{\vop}{\mrm{vop}}
\nc{\fr}{{\on{fr}}}
\nc{\Ar}{\mrm{Ar}}
\nc{\Vfr}{\on{Vfr}}
\nc{\frdiff}{{\on{frdiff}}}
\nc{\frGys}{\on{frGys}}
\nc{\SHfr}{\SH^{\fr}}
\nc{\SHfrdiff}{\SH^{\frdiff}}
\nc{\SHfrGys}{\SH^{\frGys}}
\nc{\InftyCat}{\infty\textnormal{-}\mrm{Cat}}
\nc{\TriCat}{\mathrm{TriCat}}
\nc{\Cat}{\mathrm{1\textnormal{-}Cat}}
\nc{\Th}{\on{Th}}
\def\G{\bG}
\nc{\CMon}{\mrm{CMon}{}}
\nc{\CAlg}{\mrm{CAlg}{}}
\nc{\MGL}{\mrm{MGL}}
\nc{\PMGL}{\mathrm{PMGL}}
\nc{\KGL}{\mrm{KGL}}
\nc{\kgl}{\mrm{kgl}}
\nc{\MSL}{\mrm{MSL}}
\nc{\MSp}{\mrm{MSp}}
\nc{\Seg}{\mrm{Seg}{}}
\nc{\Tw}{\mrm{Tw}}
\nc{\sslash}{/\mkern-6mu/}
\nc{\PrL}{\mrm{Pr}^\mrm{L}}
\nc{\PrR}{\mrm{Pr}^\mrm{R}}
\nc{\pr}{\mrm{pr}}
\nc{\efr}{\mrm{efr}}
\nc{\nfr}{\mrm{nfr}}
\nc{\dfr}{\mrm{fr}}
\nc{\tfr}{\mrm{tfr}}
\nc{\Vect}{\mathcal{V}\mrm{ect}}
\nc{\sVect}{\mrm{sVect}}
\nc{\fix}{\mrm{fix}}
\nc{\Hilb}{\mathrm{Hilb}}
\nc{\flci}{\mathrm{flci}}
\nc{\Isom}{\mathrm{Isom}}
\nc{\GL}{\mathrm{GL}}
\nc{\BGL}{\mathrm{BGL}}
\nc{\SL}{\mathrm{SL}}
\nc{\Sp}{\mathrm{Sp}}
\nc{\fin}{\mathrm{fin}}
\nc{\cl}{\mathrm{cl}}
\nc{\cn}{\mathrm{cn}}
\nc{\sm}{\mathrm{sm}}
\nc{\heart}{\heartsuit}
\renc{\o}{\mrm{or}}
\nc{\GW}{\mrm{GW}}
\nc{\ev}{\mrm{ev}}
\nc{\FSYN}{\mathcal{FS}\mrm{yn}}
\nc{\FFLAT}{\mathcal{FF}\mrm{lat}}
\nc{\mrk}{\mrm{mrk}}%
\nc{\FFmrk}{\mathcal{FF}\mrm{lat}^{\mrk}}
\nc{\FFnu}{\mathcal{FF}\mrm{lat}^{\mrm{nu}}}
\nc{\FFbas}{\mathcal{FF}\mrm{lat}^{\mrm{bas}}}
\nc{\FQSM}{\mathcal{FQS}\mathrm{m}}
\nc{\Quot}{\mathrm{Quot}}    
\nc{\COH}{\mathcal{C}\mathrm{oh}}
\nc{\gr}{\mrm{gr}}
\nc{\FGORIS}{\mathcal{FG}\mrm{or}\mathcal{I}\mrm{s}}
\nc{\VectOr}{\mathcal{V}\mrm{ect}\mathcal{O}\mrm{r}}
\let\phi\varphi

\let\emptyset\varnothing
\nc{\Nis}{\mrm{Nis}}
\nc{\sPre}{\mrm{sPre}}

\nc{\robber}{\mathcal{R}}%
\nc{\mv}{\mrm{mv}}
\nc{\const}{\mrm{const}}
\nc{\robbermv}{\robber^{\mv}}
\nc{\robberconst}{\robber^{\const}}
\nc{\robbernot}{\robber_0}%
\nc{\sectionmv}{i^{\mv}}%
\nc{\sectionconst}{i^{\const}}%
\nc{\sectionnot}{i}%
\nc{\st}{\mathrm{st}}
\nc{\Tor}{\mrm{Tor}}
\newcommand{\mm}{\mathfrak{m}}

\nc{\inftyCat}{\term{$\infty$-category}}
\nc{\inftyCats}{\term{$\infty$-categories}}

\nc{\inftyOneCat}{\term{$(\infty,1)$-category}}
\nc{\inftyOneCats}{\term{$(\infty,1)$-categories}}

\nc{\inftyGrpd}{\term{$\infty$-groupoid}}
\nc{\inftyGrpds}{\term{$\infty$-groupoids}}

\nc{\inftyTop}{\term{$\infty$-topos}}
\nc{\inftyTops}{\term{$\infty$-toposes}}

\nc{\inftyTwoCat}{\term{$(\infty,2)$-category}}
\nc{\inftyTwoCats}{\term{$(\infty,2)$-categories}}